\PassOptionsToPackage{hypertexnames=false,hyperfootnotes=false}{hyperref}
\documentclass[submission]{dmtcs-episciences}

\usepackage[T1]{fontenc}
\usepackage{microtype}
\usepackage{amsmath,amssymb,amsthm,mathtools}
\usepackage{graphicx}
\newcommand{\bigplus}{
  \mathop{\vcenter{\hbox{\scalebox{2.15}{$+$}}}}\displaylimits
}
\newcommand{\plus}{
  \mathop{\vcenter{\hbox{\scalebox{1.5}{$+$}}}}\displaylimits
}
\usepackage{float}
\usepackage[dvipsnames]{xcolor}
\usepackage{tikz}
\usetikzlibrary{arrows.meta,backgrounds,calc,fit,positioning,shapes.geometric}
\usepackage[round]{natbib}
\hypersetup{hidelinks}

\newtheorem{theorem}{Theorem}[section]
\newtheorem{proposition}[theorem]{Proposition}
\newtheorem{lemma}[theorem]{Lemma}
\newtheorem{corollary}[theorem]{Corollary}
\theoremstyle{definition}

\newcommand{\M}{\mathcal{M}} 
\newcommand{\R}{\mathcal{R}_{\M}} 
\newcommand{\T}{\mathcal{T}} 
\newcommand{\J}{\mathcal{J}} 
\newcommand{\Q}{\mathcal{Q}} 
\newcommand{\diam}{\operatorname{diam}} 
\newcommand{\Mob}{\operatorname{Mob}_{\mu}} 
\newcommand{\C}{\mathcal{C}} 

\author[Fres\'an-Figueroa et al.]{
  Juli\'an A. Fres\'an-Figueroa\affiliationmark{1}
  \and Il\'an A. Goldfeder\affiliationmark{2}
  \and Aldo Lozano-Pi\~na\affiliationmark{3} \thanks{Corresponding author. Email: \texttt{alozano@cua.uam.mx}.}
}

\affiliation{
  Departamento de Matem\'aticas Aplicadas y Sistemas,
  Universidad Aut\'onoma Metropolitana, Unidad Cuajimalpa, Mexico\\
  Departamento de Matem\'aticas,
  Universidad Aut\'onoma Metropolitana, Unidad Iztapalapa, Mexico\\
  Posgrado en Ciencias (Matem\'aticas),
  Universidad Aut\'onoma Metropolitana, Unidad Iztapalapa, Mexico
}

\title[Reconfiguration of mutual-visibility sets]{Reconfiguration graphs of mutual-visibility sets}

\keywords{mutual visibility, mobile mutual-visibility, reconfiguration graph, token sliding, daisy cube, complete bipartite graph, tree, cycle.}

\begin{document}
\maketitle

\begin{abstract}
A set $S\subseteq V(G)$ is a mutual-visibility set if every pair of vertices in $S$ is joined by a geodesic whose internal vertices do not belong to $S$.
We define a reconfiguration graph whose vertices are the mutual-visibility sets of $G$.
Two configurations are adjacent if one is obtained from the other by adding or removing one vertex or by sliding one token along an edge of $G$.
The subgraphs induced by configurations of a fixed cardinality are called layers.
The spanning subgraph formed by token-addition and token-removal edges is a daisy cube.
We study the effects of token-sliding edges on distances and connectivity to this skeleton.
For a connected host of order $n\ge3$, the vertex connectivity of the reconfiguration graph lies between $3$ and $n$, and both bounds are attained.
We determine several parameters of $\R(K_n)$ and describe the behavior of the model for disconnected hosts and for isometric or convex subgraphs.
We also give a component-based criterion for mobile mutual visibility.
For complete bipartite graphs, we use a description of feasible configurations to determine connectivity and distance bounds for fixed-cardinality layers.
For trees, we show that every component of a layer of cardinality at least three is isomorphic to a Cartesian product of subtrees, obtaining exact metric formulas and characterizing the connected layers.
Finally, we prove that every layer of a cycle is connected.
\end{abstract}

\section{Introduction}

Mutual visibility in graphs was introduced by \citet{DiStefano2022}.
A set $S \subseteq V(G)$ is a \emph{mutual-visibility set} if, for every pair of distinct vertices $u,v \in S$, there exists a $uv$-geodesic whose internal vertices do not belong to $S$.
We denote the family of these sets by $\M(G)$ and put $\M_k(G)=\{S\in\M(G):\lvert S\rvert=k\}$.
The maximum cardinality of a set in $\M(G)$ is the \emph{mutual-visibility number} $\mu(G)$.
Variants and structural aspects have been studied by \citet{CiceroneEtAl2023}, \citet{KorzeVesel2024}, \citet{BujtasKlavzarTian2025} and \citet{TonnyShikhi2026}, among others.

Extremal parameters give the largest feasible cardinality; reconfiguration studies how feasible solutions can be transformed into one another through allowed local operations.
For an introduction to reconfiguration problems, see \citet{Nishimura2018}.

Here, configurations are the mutual-visibility sets of a graph $G$.
We allow token addition/removal (TAR) and token sliding (TS). A TS move transfers a token to an unoccupied adjacent vertex.
The resulting graph is denoted by $\R(G)$; we call $G$ the \emph{host graph}.
The configurations of cardinality $k$ induce the $k$-th layer, denoted by $\R^k(G)$.
Only token slides occur within layers, so $\R^k(G)$ is an induced subgraph of the $k$-token graph of $G$ \citep{FabilaMonroyEtAl2012}.
A mixed reconfiguration model involving the same operations was considered for domination in \citet{MessingerPipes2026}.

Token sliding also appears in mobile mutual visibility \citep{DettlaffEtAl2025,ShallcrossEtAl2025}.
The two models ask different questions.
Mobility requires a sequence starting from one configuration in which the tokens collectively visit every vertex of the host graph, whereas reconfiguration concerns reachability between two prescribed configurations.
We characterize mobility through a covering condition on a connected component of a layer.

We first study the full graph $\R(G)$.
Heredity identifies its TAR skeleton as a daisy cube.
We examine how TS edges affect distances and connectivity and study disconnected hosts as well as isometric and convex subgraphs, and determine the parameters for complete graphs.

We then study fixed-cardinality layers.
For complete bipartite graphs, we use feasible types to determine connectivity of fixed-cardinality layers and to bound distances within connected layers.
Earlier results on maximal sets and visibility polynomials provide context for the accompanying enumeration \citep{BresarYero2024,BujtasKlavzarTian2025}.
For trees, active branching vertices act as barriers, leading to a Cartesian-product decomposition of each component, giving exact order, size, and diameter formulas.
Every layer of a cycle is connected.
Our main contributions are structural results about reconfiguration, together with some enumeration formulas; known extremal and mobile values are cited as inputs or recovered structurally.

\section{Preliminaries and the reconfiguration model}
\label{sec:preliminaries}

All graphs are finite, simple, undirected, and connected, unless stated otherwise.
We follow \citet{diestel2024graph} for standard graph terminology and notation.
For a connected graph $H$, we write $d_H$ for its distance function and $\diam(H)$ for its diameter.

Equivalently, a set $S\subseteq V(G)$ belongs to $\M(G)$ if every pair of distinct vertices $u,v\in S$ is joined by a $uv$-geodesic $P$ such that $V(P)\cap S=\{u,v\}$.
We will repeatedly use the following hereditary property.

\begin{proposition}[\citealp{CiceroneEtAl2023}]
\label{prop:cited-heredity}
If $S$ is a mutual-visibility set of a graph $G$ and $S' \subseteq S$, then $S'$ is also a mutual-visibility set of $G$.
\end{proposition}

Consequently, $\M_k(G) \neq \emptyset$ for every $0 \le k \le \mu(G)$.
The \emph{mutual-visibility reconfiguration graph} $\R(G)$ has vertex set $\M(G)$, and two distinct configurations $S,T \in \M(G)$ are adjacent if either $\lvert S \triangle T \rvert = 1$, or there exist vertices $x,y \in V(G)$ such that $S \setminus T = \{x\}$ and $T \setminus S = \{y\}$, with $xy \in E(G)$.
Moves of the first type are \emph{token addition/removal} (TAR) move; moves of the second type are \emph{token-sliding} (TS) moves from $x$ to its unoccupied neighbor $y$ (Figure~\ref{fig:moves}).
For $0 \le k \le \mu(G)$, the subgraph of $\R(G)$ induced by $\M_k(G)$ is denoted by $\R^k(G)$ and is called the $k$-th layer.
Thus TAR moves join distinct layers, and TS moves stay within a layer. We call a move \emph{admissible} if the resulting configuration belongs to $\M(G)$.

\begin{figure}[h]
\centering
\begin{tikzpicture}[
  v/.style={circle,draw,minimum size=3mm,inner sep=0pt},
  token/.style={v,fill=red!65},
  move/.style={-{Latex[length=2.4mm]},thick,blue!55!black}
]
  \node[token] (a1) at (0,0) {};
  \node[token] (a2) at (1,0) {};
  \node[v]     (a3) at (.5,.82) {};
  \draw (a1)--(a2)--(a3)--(a1);
  \node[token] (b1) at (3,0) {};
  \node[token] (b2) at (4,0) {};
  \node[token] (b3) at (3.5,.82) {};
  \draw (b1)--(b2)--(b3)--(b1);
  \draw[move] (1.35,.70)--(2.65,.70) node[midway,above]{add};
  \draw[move] (2.65,.40)--(1.35,.40) node[midway,below]{remove};
  \node at (2,-.48) {TAR move};

  \node[token] (c1) at (6,0) {};
  \node[token] (c2) at (7,0) {};
  \node[v]     (c3) at (6.5,.82) {};
  \draw (c1)--(c2)--(c3)--(c1);
  \node[v]     (d1) at (9,0) {};
  \node[token] (d2) at (10,0) {};
  \node[token] (d3) at (9.5,.82) {};
  \draw (d1)--(d2)--(d3)--(d1);
  \draw[move] (7.35,.40)--(8.65,.40) node[midway,above]{slide};
  \node at (8,-.48) {TS move};
\end{tikzpicture}
\caption{The two move types in $\R(G)$. Red vertices carry tokens. Every displayed configuration belongs to $\M(G)$. A fixed-cardinality layer contains only TS moves.}
\label{fig:moves}
\end{figure}
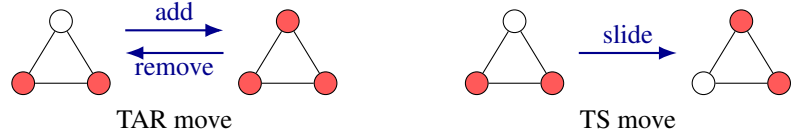

For $0 \le k \le \lvert V(G) \rvert$, the vertices of the \emph{$k$-token graph} $\T_k(G)$ are the $k$-subsets of $V(G)$; two such sets $S$ and $T$ are adjacent
if and only if $S \triangle T = \{x,y\}$, with $\{x,y\} \in E(G)$.
Token graphs were introduced by \citet{FabilaMonroyEtAl2012}, who proved the following result.

\begin{theorem}[\citealp{FabilaMonroyEtAl2012}]
\label{thm:cited-token-connectivity}
Let $G$ be a connected graph, and let $1 \le k < \lvert V(G)\rvert$.
Then $\T_k(G)$ is connected with diameter at least $k(\diam(G)-k+1)$ and at most $k\diam(G)$.
\end{theorem}

The cases $k=0$ and $k=\lvert V(G)\rvert$ each give a single vertex.
By definition, \(\R^k(G)\) is the subgraph of \(\T_k(G)\) induced by \(\M_k(G)\).

The model also admits a hypercube representation.
The \emph{$n$-dimensional hypercube} $\Q_n$ has vertex set $\{0,1\}^n$, with two vectors adjacent if they differ in exactly one coordinate.
Let $n = \lvert V(G) \rvert$ and identify each subset of $V(G)$ with its incidence vector in $\{0,1\}^n$.
Under this identification, the TAR edges of $\R(G)$ form the subgraph of $\Q_n$ induced by the incidence vectors of the sets in $\M(G)$.
The full graph $\R(G)$ is obtained from this induced subgraph by adding the TS edges within each layer.

We next introduce mobile mutual visibility.
In the mobile mutual-visibility model, an admissible TS move is called a \emph{legal mobile move}.
A set $S \in \M_k(G)$ is a \emph{mobile mutual-visibility set} if there exists a finite sequence $S_0,S_1,\ldots,S_r$ of configurations in $\M_k(G)$, with $S_0=S$, such that consecutive configurations differ by a legal mobile move and $\bigcup_{i=0}^{r} S_i = V(G)$.
The \emph{mobile mutual-visibility number} of $G$, denoted by $\Mob(G)$, is the largest integer $k$ such that $\M_k(G)$ contains a mobile mutual-visibility set.
\citet{DettlaffEtAl2025} introduced this model. The next criterion follows from the reversibility of legal moves for mobile position problems; compare \citet{DettlaffEtAl2025} and \citet{ShallcrossEtAl2025}.

\begin{proposition}[Component criterion for mobility]
\label{prop:mobile-component}
Let $G$ be a graph, let $0\le k\le\mu(G)$, and let $S\in\M_k(G)$.
Let $\mathcal C_k(S)$ denote the connected component of $\R^k(G)$ containing $S$.
Then $S$ is a mobile mutual-visibility set if and only if $\bigcup_{X \in V(\mathcal C_k(S))} X = V(G)$.
\end{proposition}

\begin{proof}
A walk in $\mathcal C_k(S)$ corresponds to a sequence of configurations in $\M_k(G)$ in which consecutive configurations differ by a legal mobile move.

Suppose that $S$ is a mobile mutual-visibility set.
Let $S_0,S_1,\ldots,S_r$ be a witnessing sequence, so $S_0=S$ and $\bigcup_{i=0}^r S_i=V(G)$.
It determines a walk starting at $S$.
Hence every $S_i$ belongs to $\mathcal C_k(S)$, and
\begin{displaymath}
    V(G) = \bigcup_{i=0}^r S_i \subseteq \bigcup_{X \in V(\mathcal C_k(S))} X \subseteq V(G).
\end{displaymath}
This proves the required covering condition.

Conversely, assume the covering condition.
Since the component is finite and connected, there is a walk $(S_0,S_1,\ldots,S_r)$ starting at $S$ and visiting every vertex of $\mathcal C_k(S)$.
Each step of this walk corresponds to a legal mobile move, and
\begin{displaymath}
    \bigcup_{i=0}^r S_i = \bigcup_{X \in V(\mathcal C_k(S))} X = V(G).
\end{displaymath}
Thus, $S$ is a mobile mutual-visibility set.
\end{proof}

\begin{corollary}
\label{cor:connected-layer-mobile}
Let $G$ be a graph and let $0 \le k \le \mu(G)$.
If $\R^k(G)$ is connected, then every member of $\M_k(G)$ is mobile if and only if every vertex of $G$ belongs to at least one member of $\M_k(G)$.
\end{corollary}

Proposition~\ref{prop:mobile-component} connects mobility, which asks whether a component covers the host, with reachability between configurations.

In a connected graph $G$, every vertex subset of cardinality at most two is a mutual-visibility set.
It follows that $\lvert \M_k(G) \rvert = \binom{\lvert V(G) \rvert}{k}$ for $k \in \{0,1,2\}$.
Therefore, $\R^0(G) \cong K_1$, $\R^1(G) \cong G$, and $\R^2(G) = \T_2(G)$.
In particular, these layers are connected.
If $G$ has at least two vertices, the configurations in $\R^2(G)$ also cover $V(G)$, and Corollary~\ref{cor:connected-layer-mobile} gives $2 \le \Mob(G) \le \mu(G)$.
Thus only layers of cardinality at least three can be disconnected.
We first study the full reconfiguration graph, then layers for specific families.

\section{Structural properties of the full reconfiguration graph}
\label{sec:full-structural}

The full graph combines TAR moves between layers with TS moves within them.
We show that its TAR skeleton is a daisy cube and derive results on distances, enumeration, and vertex connectivity, and examine complete and disconnected hosts and isometric and convex subgraphs.

\subsection{The TAR skeleton and metric bounds}

Let $G$ be a graph of order $n$.
The \emph{TAR skeleton} of $\R(G)$, denoted by $\mathcal{R}_{\mathrm{TAR}}(G)$, is the spanning subgraph obtained by deleting all TS edges.
Fix an ordering $v_1,\ldots,v_n$ of $V(G)$ and identify each subset $S\subseteq V(G)$ with its incidence vector $\chi_S \in \{0,1\}^n$, where $(\chi_S)_i = 1$ if $v_i \in S$ and $(\chi_S)_i = 0$ otherwise.
Thus subsets of $V(G)$ are identified with vertices of $\Q_n$.

For $X \subseteq \{0,1\}^n$, let $\Q_n(X)$ denote the subgraph of $\Q_n$ induced by the vectors $y = (y_1,\ldots,y_n) \in \{0,1\}^n$ for which there exists $x = (x_1,\ldots,x_n) \in X$ satisfying $y_i \le x_i$ for every $i \in \{1,\ldots,n\}$.
A graph of this form is called a \emph{daisy cube}.
A subgraph $H$ of $G$ is \emph{isometric} if $d_H(u,v)=d_G(u,v)$ for all $u,v\in V(H)$.
A \emph{partial cube} is a graph isomorphic to an isometric subgraph of a hypercube.

\begin{proposition}[\citealp{KlavzarMollard2019}]
\label{prop:cited-daisy-partial-cube}
If $X \subseteq \{0,1\}^n$, then $\Q_n(X)$ is a partial cube.
\end{proposition}

\begin{theorem}[The TAR skeleton]
\label{thm:tar-isometric}
Let $G$ be a graph of order $n$, and let $\widehat{\M}(G)$ be the family of inclusion-maximal mutual-visibility sets of $G$.
Under the incidence-vector identification, $\mathcal{R}_{\mathrm{TAR}}(G) = \Q_n \left(\widehat{\M}(G)\right)$.
In particular, $\mathcal{R}_{\mathrm{TAR}}(G)$ is a daisy cube.
Moreover, for all $S,T \in \M(G)$, we have $d_{\mathcal{R}_{\mathrm{TAR}}(G)} (S,T) = \lvert S \triangle T \rvert$.
Consequently, $\mathcal{R}_{\mathrm{TAR}}(G)$ is an isometric subgraph of $\Q_n$ and, in particular, a partial cube.
\end{theorem}

\begin{proof}
Heredity makes $\M(G)$ downward closed, and each of its members lies in an inclusion-maximal member.
Thus its incidence vectors form the downset generated by $\widehat{\M}(G)$; TAR adjacency agrees with hypercube adjacency.
This proves the daisy-cube description, and Proposition~\ref{prop:cited-daisy-partial-cube} gives the partial-cube assertion.

For $S,T\in\M(G)$, remove the vertices of $S\setminus T$ one at a time and then add those of $T\setminus S$.
Every intermediate set is a subset of $S$ or $T$, so the resulting TAR path has length $\lvert S \triangle T \rvert$.
Any hypercube path must change at least these $\lvert S \triangle T \rvert$ coordinates.
This proves the distance formula and isometry.
\end{proof}

The TS edges may shorten distances relative to the TAR skeleton.
The next bounds are expressed in terms of $\lvert S\setminus T\rvert$ and $\lvert T\setminus S\rvert$.

\begin{proposition}[Metric bounds]
\label{prop:full-metric-bounds}
Let $G$ be a graph of order $n$.
For all $S,T \in \M(G)$, the inequality $\max \left\{\lvert S \setminus T \rvert,\lvert T \setminus S \rvert\right\} \le d_{\R(G)} (S,T) \le \lvert S \triangle T \rvert$ holds.
If $S \subseteq T$, then $d_{\R(G)} (S,T) = \lvert T \setminus S \rvert$.
In particular, $d_{\R(G)}(\emptyset,S)=\lvert S\rvert$ and $\operatorname{ecc}_{\R(G)}(\emptyset)=\mu(G)$, where $\operatorname{ecc}$ denotes eccentricity.
Consequently, $\mu(G) \le \diam(\R(G)) \le \min \{n,2\mu(G)\}$.
\end{proposition}

\begin{proof}
Each move removes at most one vertex of $S\setminus T$ and inserts at most one vertex of $T\setminus S$.
Hence every $ST$-path has length at least $\max\{\lvert S \setminus T \rvert,\lvert T \setminus S \rvert\}$.
The preceding theorem gives a TAR path of length $\lvert S \triangle T\rvert$, proving the upper bound.
If $S\subseteq T$, both bounds equal $\lvert T \setminus S \rvert$.
Taking $S = \emptyset$ gives $\operatorname{ecc}_{\R(G)}(\emptyset)=\mu(G)$ and thus $\operatorname{diam}(\R(G))\ge\mu(G)$.
Finally, $\lvert S \triangle T \rvert\le\min\{n,\lvert S \rvert + \lvert T \rvert \}\le\min\{n,2\mu(G)\}$.
\end{proof}

Every TS edge can be replaced by a two-edge path in the TAR skeleton.
Indeed, if $S$ and $T$ are TS-adjacent, then $S \cap T \in \M(G)$, and both $S$ and $T$ are TAR-adjacent to $S \cap T$.
Hence these three configurations form a $3$-cycle in $\R(G)$.

\subsection{Counting configurations, TAR edges, and local structure}

The visibility polynomial of $G$ was introduced by \citet{BujtasKlavzarTian2025}.
In our notation, it is given by $P_G(z) = \sum_{k=0}^{\mu(G)} m_k z^k$, where $m_k = \lvert \M_k(G) \rvert$.

For a daisy cube $H=\Q_d(X)$, its \emph{cube polynomial} is $C_H(z)=\sum_{j\ge 0} c_j(H)z^j$, where $c_j(H)$ is the number of induced subgraphs of $H$ isomorphic to $\Q_j$. Writing $0^d$ for the all-zero vector, its weight enumerator with respect to $0^d$ is $W_{H,0^d}(z)=\sum_{v\in V(H)} z^{d_H(0^d,v)}$.
We recall the cube-polynomial identity from the following result.

\begin{corollary}[\citealp{KlavzarMollard2019}]
\label{cor:cited-daisy-polynomial}
If $H=\Q_d(X)$ is a daisy cube, then $C_H(z)=W_{H,0^d}(z+1)$.
\end{corollary}

By Theorem~\ref{thm:tar-isometric}, $W_{\mathcal R_{\mathrm{TAR}}(G),\emptyset}(z)=P_G(z)$, so Corollary~\ref{cor:cited-daisy-polynomial} gives $C_{\mathcal R_{\mathrm{TAR}}(G)}(z)=P_G(z+1)$.
Since $c_0(H)=\lvert V(H)\rvert$ and $c_1(H)=\lvert E(H)\rvert$, the constant and linear coefficients give the two identities below.
We also give a direct counting proof.

\begin{proposition}
\label{prop:full-counting}
The order and number of TAR edges are $\lvert V(\R(G)) \rvert = \sum_{k=0}^{\mu(G)} m_k = P_G(1)$ and $\lvert E(\mathcal R_{\mathrm{TAR}}(G)) \rvert = \sum_{k=1}^{\mu(G)} k m_k = P'_G(1)$.
\end{proposition}

\begin{proof}
Every configuration in $\M(G)$ belongs to exactly one
of the families $\M_k(G)$, according to its cardinality.
Therefore, $\lvert V(\R(G)) \rvert = \sum_{k=0}^{\mu(G)} \lvert \M_k(G) \rvert = P_G(1)$.
Each configuration in $\M_k(G)$ has $k$ TAR neighbors in $\M_{k-1}(G)$, one per token removal.
Counting from the larger endpoint counts each TAR edge once.
Hence, $\lvert E(\mathcal R_{\mathrm{TAR}}(G)) \rvert = \sum_{k=1}^{\mu(G)} k \lvert \M_k(G) \rvert = P'_G(1)$.
\end{proof}

For a set $S \in \M(G)$, the neighbors of $S$ fall into three disjoint classes: token removals, token additions, and token slides.
Heredity makes each of the $\lvert S \rvert$ possible removals admissible.
To count the remaining two types of neighbors, define $s_G(S) = \left\lvert \{(x,y) \in S \times (V(G) \setminus S) \colon xy \in E(G), (S \setminus \{x\}) \cup \{y\} \in \M(G)\} \right\rvert$ and $a_G(S) = \left\lvert \{v \in V(G) \setminus S \colon S \cup \{v\} \in \M(G) \} \right\rvert$.
These count admissible slides and additions, respectively.
Therefore, $\deg_{\R(G)} (S) = \lvert S \rvert + a_G(S) + s_G(S)$.
In particular, $\deg_{\R(G)}(\emptyset)=n$.

\subsection{Vertex connectivity}

\begin{theorem}[Vertex connectivity]
\label{thm:full-vertex-connectivity}
Let $G$ be a connected graph of order $n\ge3$.
Then the reconfiguration graph satisfies $3\le\kappa(\R(G))\le n$.
Both bounds are sharp.
\end{theorem}

\begin{proof}
The empty configuration has precisely $n$ neighbors, namely the singleton configurations.
Deleting these neighbors isolates $\emptyset$, and therefore $\kappa(\R(G)) \le n$.

To prove the lower bound, let $F \subseteq V(\R(G))$ with $\lvert F \rvert \le 2$.
We show that $\R(G)-F$ is connected.
Let $H$ be the subgraph of $\R(G)-F$ induced by the remaining configurations of cardinality at most two.
We first prove that $H$ is connected and then show that every remaining larger configuration has a path to $H$.

Suppose first that $\emptyset \notin F$.
Every remaining singleton configuration is adjacent to $\emptyset$.
A remaining configuration $\{u,v\}$ of cardinality two is adjacent to one of these singletons unless both $\{u\}$ and $\{v\}$ belong to $F$.
If both belong to $F$, then $F = \{\{u\},\{v\}\}$.
Since $G$ is connected and $n \ge 3$, there exist $x \in \{u,v\}$ and $w \notin \{u,v\}$ such that $xw \in E(G)$.
Sliding the token from $x$ to $w$ yields the configuration $U = (\{u,v\} \setminus \{x\}) \cup \{w\}$.
Neither $U$ nor $\{w\}$ belongs to $F$.
Thus, $\{u,v\}$ is joined to $\emptyset$ by the path $\{u,v\},U,\{w\},\emptyset$, which lies entirely in $H$.
Thus every vertex reaches $\emptyset$, so $H$ is connected.

Suppose now that $\emptyset\in F$.
Before deleting the configurations in $F$, the TAR edges between $\M_1(G)$ and $\M_2(G)$ form the graph obtained by subdividing every edge of $K_n$ once.
Singletons correspond to original vertices, and $\{u,v\}$ subdivides the edge joining $\{u\}$ and $\{v\}$.

Since $\emptyset \in F$ and $\lvert F \rvert \le 2$, at most one vertex of this subdivision graph is deleted.
If none is deleted, the graph remains connected.
Deleting a configuration of cardinality two leaves a subdivision of $K_n$ with one edge removed, which is connected because $n \ge 3$.
Deleting a singleton configuration leaves a subdivision of $K_{n-1}$ together with pendant vertices, each attached to a remaining singleton configuration.
This graph is also connected.
In each case, the remaining TAR edges connect $H$.

Now let $S \in V(\R(G)) \setminus F$ with $\lvert S \rvert \ge 3$.
Every configuration $X$ of cardinality at least three has at least three distinct TAR neighbors of cardinality $\lvert X \rvert - 1$, namely the configurations $X \setminus \{x\}$ with $x \in X$.
Heredity makes these configurations feasible.
Since $\lvert F \rvert \le 2$, at least one of these neighbors lies outside $F$.
Repeating this deletion step produces a TAR path from $S$ to $H$ in $\R(G)-F$.

Finally, let $A$ and $B$ be any two vertices of $\R(G)-F$.
Each has a path to $H$, possibly of length zero.
Connecting these paths through $H$ gives an $AB$-walk in $\R(G)-F$.
Thus, $\R(G)-F$ is connected.
Since $F$ was arbitrary, $\kappa(\R(G))\ge3$.

Both bounds are attained.
For the lower bound, let $G = P_n$ with $n \ge 3$, and let $v_1$ be an end vertex with neighbor $v_2$.
A mutual-visibility set of a path has at most two vertices.
The configuration $\{v_1,v_2\}$ has two TAR neighbors and exactly one TS neighbor, obtained by sliding the token on $v_2$ to $v_3$.
Hence, its degree is three, and the lower bound gives $\kappa(\R(P_n)) = 3$.

For the upper bound, let $G = K_n$.
Every subset of $V(K_n)$ is a mutual-visibility set, so the TAR skeleton of $\R(K_n)$ is $\Q_n$.
The spanning $n$-connected subgraph $\Q_n$ gives $\kappa(\R(K_n))\ge n$.
Therefore $\kappa(\R(K_n)) = n$.
\end{proof}

A natural question is whether $\R(G)$is maximally connected for every
connected $G$, that is, whether $\kappa(\R(G))=\delta(\R(G))$ holds, where $\delta$ denotes minimum degree?
The theorem settles this for its two extremal families.

\subsection{The complete graph}

For complete graphs, every configuration is feasible and the parameters in Proposition~\ref{prop:full-complete-graph} admit closed formulas.
For integers $n$ and $k$ with $0 \le k \le n$, the vertices of the \emph{Johnson graph} $\J(n,k)$ are the $k$-subsets of an $n$-element set.
Adjacency in $\J(n,k)$ means $\lvert S\triangle T\rvert=2$, a single-element replacement.

\begin{proposition}
\label{prop:full-complete-graph}
Let $n\ge3$ and $0\le k\le n$.
The layer $\R^k(K_n)$ is isomorphic to $\J(n,k)$, and $\R(K_n)$ is obtained from $\Q_n$ by adding the edges of $\J(n,k)$ within each level $k$.
Moreover, for configurations $S,T\subseteq V(K_n)$,
\begin{enumerate}
    \item $\lvert V(\R(K_n)) \rvert = \lvert \M(K_n) \rvert = 2^n$.
    \item $\lvert E(\R(K_n)) \rvert = n(n+3)2^{n-3}$.
    \item If $\lvert S \rvert = k$, then $\deg_{\R(K_n)} (S) = n+k(n-k)$.
    \item $d_{\R(K_n)} (S,T) = \max \{\lvert S \setminus T \rvert,\lvert T \setminus S \rvert\}$.
    \item $\diam(\R(K_n)) = n$.
    \item The radius is $\operatorname{rad}(\R(K_n))=\lceil n/2\rceil$.
    \item The center of $\R(K_n)$ consists of the configurations of cardinality $\lfloor n/2 \rfloor$ and $\lceil n/2 \rceil$.
    \item $\kappa(\R(K_n)) = n$.
\end{enumerate}
\end{proposition}

\begin{proof}
Every subset of $V(K_n)$ is a mutual-visibility set. 
Consequently, the TAR edges form $\Q_n$, whereas the TS edges in level $k$ form $\J(n,k)$.
This proves the structural description and (1).

To prove (2) and (3), let $S$ be a configuration of cardinality $k$.
Each TS neighbor of $S$ is uniquely obtained by choosing a vertex
$x\in S$ and a vertex $y\in V(K_n)\setminus S$, and replacing $x$
with $y$. Hence, $S$ has exactly $k(n-k)$ TS neighbors. It also has
$k$ deletion neighbors and $n-k$ addition neighbors. Therefore, $\deg_{\R(K_n)}(S)=k+(n-k)+k(n-k)=n+k(n-k)$, proving (3).

The TAR skeleton has $n2^{n-1}$ edges.
The sum $\sum_{k=0}^{n}\binom{n}{k}k(n-k)=n(n-1)2^{n-2}$ counts triples $(S,x,y)$ with $x\in S$ and $y\notin S$: choose the ordered pair $(x,y)$ in $n(n-1)$ ways and the remaining memberships in $2^{n-2}$ ways.
Each TS edge is counted twice.
Therefore $\lvert E(\R(K_n)) \rvert = n2^{n-1}+n(n-1)2^{n-3}=n(n+3)2^{n-3}$, proving (2).

For the distance formula, let $A = S \setminus T$ and $B = T \setminus S$, and put $a = \lvert A \rvert$ and $b = \lvert B \rvert$.
Consider an arbitrary $ST$-path of length $\ell$.
Every vertex of $A$ must be removed.
Since each move removes at most one vertex of $A$, we have $\ell \ge a$.
Every vertex of $B$ must be inserted; at most one is inserted per move, so $\ell\ge b$.
Therefore, $\ell \ge \max \{a,b\}$.
Thus $d_{\R(K_n)}(S,T)\ge\max\{a,b\}$.

For the reverse inequality, write
$A=\{x_1,\ldots,x_a\}$ and $B=\{y_1,\ldots,y_b\}$, and let
$r=\min\{a,b\}$. For each $1\le  i\le  r$, slide the token from
$x_i$ to $y_i$. These moves are admissible because every two distinct
vertices of $K_n$ are adjacent. After these $r$ moves, if $a>b$,
delete the remaining vertices $x_{r+1},\ldots,x_a$; if $b>a$, add
the remaining vertices $y_{r+1},\ldots,y_b$. All intermediate
configurations are feasible because every subset of
$V(K_n)$ is mutually visible. The resulting $ST$-path has length
$r+\lvert a-b\rvert=\max\{a,b\}$. Consequently,
$d_{\R(K_n)}(S,T)=\max\{\lvert S\setminus T\rvert,
\lvert T\setminus S\rvert\}$, proving (4).

We now determine the eccentricity of every configuration.
Fix $S \subseteq V(K_n)$ with $\lvert S \rvert = k$.
For every configuration $T \subseteq V(K_n)$, we have $\lvert S \setminus T \rvert \le k$ and $\lvert T \setminus S \rvert \le n-k$.
Thus, by (4), $d_{\R(K_n)}(S,T) \le \max \{k,n-k\}$.
The complement $T=V(K_n)\setminus S$ attains this bound, since $\lvert S\setminus T\rvert=k$ and $\lvert T\setminus S\rvert=n-k$.
Therefore, $\operatorname{ecc}_{\R(K_n)} (S) = \max \{k,n-k\}$.

The maximum eccentricity is $n$, attained exactly for $k \in \{0,n\}$.
Moreover, the configurations $\emptyset$ and $V(K_n)$ are at distance $n$ by (4).
Hence, $\diam(\R(K_n)) = 
n$, proving (5).

For the radius and center, minimize $\max\{k,n-k\}$ over $0\le k\le n$.
Since $k+(n-k)=n$, this maximum is at least $\lceil n/2 \rceil$.
If $n$ is even, equality holds precisely when $k=n/2$.
If $n$ is odd, equality holds precisely when $k=(n-1)/2$ or $k=(n+1)/2$.
Consequently, $\operatorname{rad} (\R(K_n)) = \lceil n/2 \rceil$, and the center consists exactly of the configurations of cardinality $\lfloor n/2 \rfloor$ or $\lceil n/2 \rceil$.
This proves (6) and (7).

Finally, Theorem~\ref{thm:full-vertex-connectivity} gives (8).
\end{proof}

Figure~\ref{fig:R_M(K4)} illustrates the structural description in the case $n=4$.
The gray TAR edges form $\Q_4$.
The colored edges
represent TS moves within the nontrivial cardinality layers: the green
and blue edges induce copies of $K_4$ in
$\mathcal{R}_{\mathcal{M}}^{1}(K_4)$ and
$\mathcal{R}_{\mathcal{M}}^{3}(K_4)$, respectively, while the orange
edges induce the Johnson graph $\J(4,2)$ in
$\mathcal{R}_{\mathcal{M}}^{2}(K_4)$.

\tikzset{
    pics/K4/.style args={#1/#2/#3/#4}{
        code = {
            \node[
                circle, 
                draw = white, 
                fill = white, 
                inner sep = 15.5pt
            ] at (0,0) {};
            \draw[black] 
                (-0.4,-0.4) -- (-0.4,0.4) -- (0.4,0.4) -- (0.4,-0.4) -- (-0.4,-0.4)
                (-0.4,-0.4) -- (0.4,0.4)
                (-0.4,0.4) -- (0.4,-0.4)
            ;
            \fill[draw = black, fill = #1] (-0.4,-0.4) circle (3.5pt);
            \fill[draw = black, fill = #2] (-0.4,0.4) circle (3.5pt);
            \fill[draw = black, fill = #3] (0.4,0.4) circle (3.5pt);
            \fill[draw = black, fill = #4] (0.4,-0.4) circle (3.5pt);
        }
    }
}

\begin{figure}[h]
\centering
\begin{tikzpicture}[scale = 0.65,
    vertex/.style = {
        circle,
        fill = black,
        inner sep = 1.8pt
    },
    edge/.style = {
        line width = 1.1pt,
        color = gray
    }
]
    \def\escala{0.82}

    \node[vertex] (v1) at (-5.4, 7) {};
    \node[vertex] (v12) at (7.3, 7) {};
    \node[vertex] (v124) at (5.7, 4) {};
    \node[vertex] (v14) at (-7.2, 4) {};
    
    \node[vertex] (v134) at (-7.2, -7.8) {};
    \node[vertex] (V) at (5.7, -7.8) {};
    \node[vertex] (v123) at (7.3, -4.8) {};
    \node[vertex] (v13) at (-5.4, -4.8) {};
    
    \node[vertex] (O) at (-2.2, 3.1) {};
    \node[vertex] (v2) at (3.4, 3.1) {};
    \node[vertex] (v24) at (2, 0.9) {};
    \node[vertex] (v4) at (-3.3, 0.9) {};
    
    \node[vertex] (v34) at (-3.3, -4) {};
    \node[vertex] (v234) at (2, -4) {};
    \node[vertex] (v23) at (3.4, -2) {};
    \node[vertex] (v3) at (-2.2, -2) {};
    
    \draw[edge]
        (v1)--(v12)--(v124)--(v14)--(v1)--(v13)--(v123)--(V)--(v134)--(v13)
        (v14)--(v134)
        (v12)--(v123)
        (V)--(v124)
    ;
    
    \draw[edge] 
        (O)--(v2)--(v24)--(v4)--(O)--(v3)--(v23)--(v234)--(v34)--(v3)
        (v4)--(v34)
        (v24)--(v234)
        (v2)--(v23)
    ;
    
    \draw[edge] (v1)--(O);
    \draw[edge] (v12)--(v2);
    \draw[edge] (v124)--(v24);
    \draw[edge] (v4)--(v14);
    
    \draw[edge] (v13)--(v3);
    \draw[edge] (v134)--(v34);
    \draw[edge] (v123)--(v23);
    \draw[edge] (V)--(v234);

    \draw[edge, Green]
        (v1) -- (v2)
        (v3) -- (v4) -- (v1)
        (v2) -- (v4)
    ;
    \draw[edge, Green, bend right = 40] (v1) to (v3);
    \draw[edge, Green, bend right = 7] (v2) to (v3);

    \draw[edge, Blue]
        (v234) -- (v134)
        (v124) -- (v123) -- (v234)
        (v134) -- (v123)
    ;
    \draw[edge, Blue, bend right = 10] (v124) to (v234);
    \draw[edge, Blue, bend right = 10] (v134) to (v124);

    \draw[edge,Dandelion]
        (v14) -- (v13) -- (v34) -- (v23) -- (v24)
        (v34) -- (v24) -- (v14)
        (v14) -- (v34)
    ;
    \draw[edge, Dandelion, bend left = 15] (v13) to (v12);
    \draw[edge, Dandelion, bend right = 85] (v13) to (v23);
    \draw[edge, Dandelion, bend left = 13] (v12) to (v23);
    \draw[edge, Dandelion, bend right = 39] (v12) to (v24);
    \draw[edge, Dandelion, bend right = 6] (v12) to (v14);
    
    \pic[scale = \escala, transform shape] at (v4) {K4 = white/white/white/red};
    \pic[scale = \escala, transform shape] at (v14) {K4 = red/white/white/red};
    \pic[scale = \escala, transform shape] at (v124) {K4 = red/red/white/red};
    \pic[scale = \escala, transform shape] at (v24) {K4 = white/red/white/red};
    
    \pic[scale = \escala, transform shape] at (v234) {K4 = white/red/red/red};
    \pic[scale = \escala, transform shape] at (V) {K4 = red/red/red/red};
    \pic[scale = \escala, transform shape] at (v123) {K4 = red/red/red/white};
    \pic[scale = \escala, transform shape] at (v23) {K4 = white/red/red/white};
    
    \pic[scale = \escala, transform shape] at (O) {K4 = white/white/white/white};
    \pic[scale = \escala, transform shape] at (v1) {K4 = red/white/white/white};
    \pic[scale = \escala, transform shape] at (v12) {K4 = red/red/white/white};
    \pic[scale = \escala, transform shape] at (v2) {K4 = white/red/white/white};
    
    \pic[scale = \escala, transform shape] at (v3) {K4 = white/white/red/white};
    \pic[scale = \escala, transform shape] at (v13) {K4 = red/white/red/white};
    \pic[scale = \escala, transform shape] at (v134) {K4 = red/white/red/red};
    \pic[scale = \escala, transform shape] at (v34) {K4 = white/white/red/red};
\end{tikzpicture}
\caption{The reconfiguration graph $\mathcal{R}_{\mathcal{M}}(K_4)$. The gray edges are TAR edges joining configurations in consecutive cardinality layers. The blue, orange, and green edges are the TS edges within the layers $\mathcal{R}_{\mathcal{M}}^{3}(K_4)$, $\mathcal{R}_{\mathcal{M}}^{2}(K_4)$, and $\mathcal{R}_{\mathcal{M}}^{1}(K_4)$, respectively. In particular, $\mathcal{R}_{\mathcal{M}}^{2}(K_4)\cong \J(4,2)$, whereas $\mathcal{R}_{\mathcal{M}}^{1}(K_4) \cong \mathcal{R}_{\mathcal{M}}^{3}(K_4)\cong K_4$.}

\label{fig:R_M(K4)}
\end{figure}
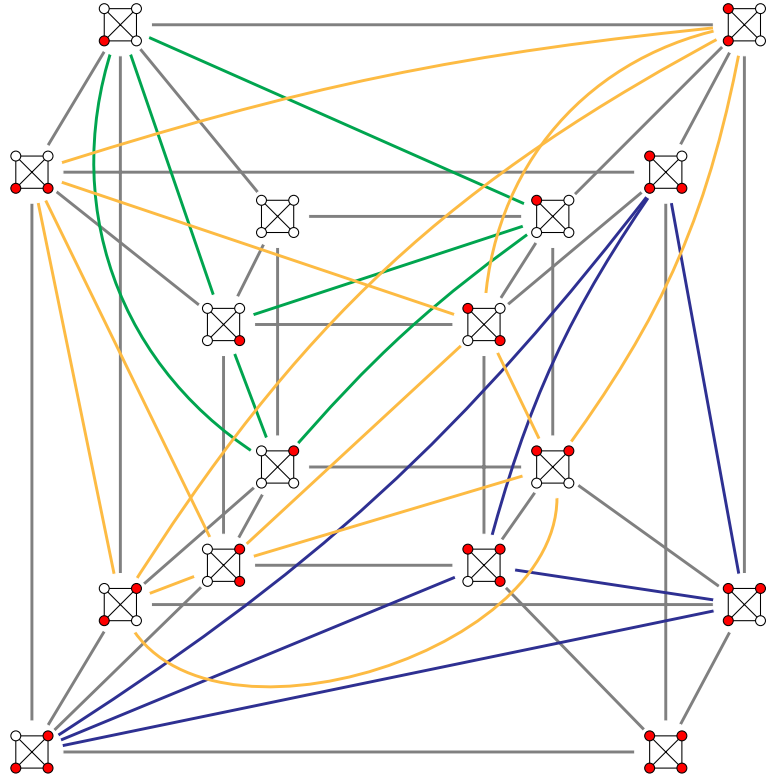

\subsection{Disconnected host graphs}

We now allow disconnected hosts.
Let $G_1,\ldots,G_c$ be the connected components of \(G\), and write $\mu_i = \mu(G_i)$.
A mutual-visibility set cannot meet distinct components, since no geodesic joins them.
We say that a configuration $S$ is \emph{supported in} a subgraph $H$ of $G$ if $S \subseteq V(H)$. Thus, every nonempty configuration is supported in a unique component of $G$.

We regard each $\R(G_i)$ as its natural copy in $\R(G)$.
These copies share the empty configuration and are otherwise pairwise disjoint.

For graphs $H_1,\ldots,H_t$, we denote their disjoint union by $\plus_{i=1}^t H_i$.
This takes vertex-disjoint copies of the $H_i$ with no edges between copies.

\begin{proposition}[Decomposition by components]
\label{prop:full-component-decomposition}
Let $G$ be a graph with connected components $G_1,\ldots,G_c$, where $c \ge 1$, and put $\mu_i = \mu(G_i)$.
Then $\mu(G) = \max_{1 \le i \le c} \mu_i$.
For every integer $k$ with $1 \le k \le \mu(G)$,
\begin{displaymath}
    \R^k(G) \cong \bigplus_{\substack{1 \le i \le c \\ k \le \mu_i}} \R^k(G_i).
\end{displaymath}
For $k=0$, the layer $\R^0(G)$ consists of the single vertex $\emptyset$.
Furthermore, $\R(G)=\bigcup_{i=1}^{c}\R(G_i)$, with $\R(G_i)\cap\R(G_j)=\{\emptyset\}$ for $i\ne j$.
Equivalently, \(\R(G)\) is obtained by identifying the empty configurations of the graphs \(\R(G_i)\) into a single vertex..
If $c\ge2$, then $\emptyset$ is a cut vertex and $\kappa(\R(G))=1$.
Moreover, if $c \ge 2$ and the components are relabeled so that $\mu_1 \ge \mu_2 \ge \cdots \ge \mu_c$, then $\diam(\R(G)) = \max \left\{D,\mu_1+\mu_2 \right\}$, where $D = \max_{1 \le i \le c} \{\diam(\R(G_i))\}$.
\end{proposition}

\begin{proof}
Every nonempty mutual-visibility set of $G$ belongs to exactly one $\M(G_i)$.
Conversely, $\M(G_i)\subseteq\M(G)$.
Therefore, $\mu(G) = \max_{1 \le i \le c} \mu_i$.
The assertion for $k=0$ is immediate.

Let $k \ge 1$ and $S \in \M_k(G_i)$.
A TS move replaces $x\in S$ by an unoccupied neighbor $y$.
Since $xy\in E(G)$, $y\in V(G_i)$.
Thus, the resulting configuration still belongs to $\M_k(G_i)$.
TS adjacency within $V(G_i)$ agrees in $G_i$ and $G$.
Consequently, $\R^k(G)$ is the disjoint union of the graphs $\R^k(G_i)$ for which $k \le \mu_i$.

A TAR move whose two endpoints are nonempty also preserves the host component supporting the configuration.
Removals leave the remaining tokens in that component; additions outside it violate mutual visibility.
Hence, any reconfiguration sequence between nonempty configurations contained in different components of $G$ must pass through $\emptyset$.
Thus the copies $\R(G_i)$ meet only at $\emptyset$.
If $c \ge 2$, deleting $\emptyset$ separates the nonempty configurations supported in different components and hence $\kappa(\R(G)) = 1$.

Now let $S$ and $T$ be supported in different components.
Every $ST$-path passes through $\emptyset$, and Proposition~\ref{prop:full-metric-bounds} gives $d_{\R(G)}(S,T) = d_{\R(G)}(S,\emptyset)+d_{\R(G)}(\emptyset,T) = \lvert S \rvert + \lvert T \rvert$.
The maximum distance across summands is $\mu_1+\mu_2$; within the $i$-th summand it is $\diam(\R(G_i))$.
Maximizing gives the formula.
\end{proof}

\subsection{Isometric and convex subgraphs}

Passing to an induced subgraph need not preserve mutual visibility: vertex deletion may alter distances and witnessing geodesics.
A subgraph $H$ of $G$ is \emph{convex} if every geodesic in $G$ between two vertices of $H$ is entirely contained in $H$.
In particular, every convex subgraph is isometric.

We recall the following results on convex subgraphs.

\begin{lemma}[\citealp{DiStefano2022}]
\label{lem:cited-convex-restriction}
Let $H$ be a convex subgraph of $G$, and let $S \subseteq V(G)$ be a mutual-visibility set of $G$.
Then $S \cap V(H)$ is a mutual-visibility set of $H$.
\end{lemma}

Lemma~\ref{lem:cited-convex-restriction} is also stated by \citet[Lemma~2.3]{BujtasKlavzarTian2025}, who credit Di Stefano; both sources state the same restriction property.

\begin{lemma}[\citealp{DiStefano2022}]
\label{lem:cited-convex-monotonicity}
Let $H$ be a convex subgraph of a graph $G$.
Then $\mu(H) \le \mu(G)$.
\end{lemma}

We extend the feasible-set inclusion to isometric subgraphs and describe induced embeddings of the full reconfiguration graph and its layers.

\begin{theorem}[Isometric embeddings]
\label{thm:full-isometric-embedding}
Let $G$ be a graph and let $H$ be an isometric subgraph of $G$.
Then, $\M(H) \subseteq \M(G)$ and the natural inclusion gives an isomorphism from $\R(H)$ onto the subgraph of $\R(G)$ induced by $\M(H)$.
The same conclusion holds for every fixed-cardinality layer.
\end{theorem}

\begin{proof}
Let $S \in \M(H)$.
For every $u,v \in S$, there is a $uv$-geodesic in $H$ whose internal vertices avoid $S$.
Since $H$ is isometric, this path is also a geodesic in $G$ and hence $S \in \M(G)$.

It remains to compare adjacency.
TAR adjacency depends only on symmetric difference, so is preserved in both directions.
TS adjacency from $\R(H)$ to $\R(G)$ is preserved because $E(H) \subseteq E(G)$.
Conversely, every isometric subgraph is induced: if $x,y \in V(H)$ and $xy \in E(G)$, then $d_G(x,y) = 1 = d_H(x,y)$ and therefore $xy \in E(H)$.
Hence, every TS edge of $\R(G)$ with both endpoints in $\M(H)$ is also present in $\R(H)$.
Thus the image is an induced subgraph.
Restricting to cardinality $k$ proves the layer statement.
\end{proof}

Configurations of $G$ supported in $H$ need not belong to $\M(H)$.
Isometry guarantees that distances agree, but it does not imply that every geodesic of $G$ with endpoints in $H$ lies in $H$.
Convexity rules out this phenomenon and yields equality.

\begin{corollary}[Convex subgraphs]
\label{cor:full-convex-embedding}
If $H$ is a convex subgraph of $G$, then $\M(H) = \{S \in \M(G) \colon S \subseteq V(H)\}$.
Consequently, $\R(H)$ is precisely the subgraph of $\R(G)$ induced by the configurations supported in $H$.
\end{corollary}

\begin{proof}
A convex subgraph is isometric, so Theorem~\ref{thm:full-isometric-embedding} gives one inclusion.
Conversely, let $S \in \M(G)$ with $S \subseteq V(H)$.
Every geodesic in $G$ between vertices of $H$ is contained in $H$.
Thus the geodesics witnessing visibility for $S$ remain in $H$ and avoid the other vertices of $S$. 
Thus, $S \in \M(H)$.
\end{proof}

\begin{corollary}[Clique-boundary deletion]
\label{cor:full-clique-boundary}
Let $G$ be a graph and let $D\subsetneq V(G)$, with external boundary $\partial D=N_G(D)\setminus D$.
If $\partial D$ induces a clique, then $G-D$ is convex in $G$.
Consequently, $\M(G-D) = \{S \in \M(G) \colon S \cap D = \emptyset\}$, and $\R(G-D)$ is the subgraph of $\R(G)$ induced by the configurations that avoid $D$.
\end{corollary}

\begin{proof}
Suppose that a geodesic with endpoints outside $D$ enters $D$.
Let $x$ and $y$ be the vertices immediately preceding and following a maximal subpath contained in $D$.
Then, $x,y \in \partial D$.
Since $\partial D$ is a clique, replacing this subpath with the edge $xy$ produces a shorter path.
Hence every geodesic between vertices of $G-D$ stays there.
Thus, $G-D$ is convex, and the conclusion follows from Corollary~\ref{cor:full-convex-embedding}.
\end{proof}

When $D=\{v\}$, this is the simplicial-vertex case, since $\partial D=N_G(v)$ is a clique.
The fact that $\R^k(H)$ is an induced subgraph of $\R^k(G)$ does not imply that distinct components remain disconnected in the larger layer: paths in $\R^k(G)$ may leave and reenter the copy.
We next study layers of complete bipartite graphs.
In this family, mutual visibility is determined by the number of selected vertices in each partite class.

\section{Complete bipartite graphs}

Let $G = K_{m,n}$ have bipartition $A \cup B$, where $\lvert A \rvert = m$ and $\lvert B \rvert = n$.
For $S \subseteq V(G)$, we say that $S$ is of type $(a,b)$ if $a = \lvert S \cap A \rvert$ and $b = \lvert S \cap B \rvert$.
The extremal value is known.

\begin{theorem}[\citealp{DiStefano2022}]
\label{thm:cited-bipartite-mu}
Let $G=K_{m,n}$ be a complete bipartite graph with $m,n\ge3$.
Then $\mu(G)=m+n-2$.
\end{theorem}

The feasible configurations of $K_{m,n}$ follow directly from its length-two geodesics.
For the balanced case, \citet[Proposition~3.2]{BujtasKlavzarTian2025} give an equivalent case description in the proof of their visibility-polynomial formula; they also note that the unbalanced case can be treated similarly.
We state the conditions uniformly for $m,n\ge2$ because they identify the feasible levels used in the layer results below.

\begin{theorem}[Set characterization]
\label{thm:Kmn-characterization}
Let $m,n\ge2$, and let $A,B$ be the partite classes of $K_{m,n}$, with $\lvert A\rvert=m$ and $\lvert B\rvert=n$.
A set $S\subseteq V(K_{m,n})$ of type $(a,b)$, where $a=\lvert S\cap A\rvert$ and $b=\lvert S\cap B\rvert$, is a mutual-visibility set if and only if the following conditions hold:
\begin{enumerate}[(i)]
    \item $a \le 1$ or $b < n$, and
    \item $b \le 1$ or $a < m$.
\end{enumerate}
\end{theorem}

\begin{proof}
Assume first that $a \ge 2$.
Any two vertices of $S \cap A$ are at distance two, and every geodesic joining them has its internal vertex in $B$.
Thus any such pair is mutually visible if and only if $B\setminus S\ne\emptyset$, that is, $b<n$.
This proves the first condition.
By symmetry, if $b \ge 2$, then the vertices of $S \cap B$ are pairwise mutually visible if and only if $a<m$.
Vertices in different partite classes are adjacent, imposing no further restriction.
The conditions are therefore necessary and sufficient.
\end{proof}

A mutual-visibility set $S$ is \emph{maximal} if no proper superset is mutually visible.
It is \emph{maximum} if $\lvert S \rvert = \mu(G)$.
Maximum sets are maximal, but the converse may fail.
The proof of \citet[Proposition~6]{BresarYero2024} considers the same patterns of maximal mutual-visibility sets.
We record them explicitly in type notation.

\begin{corollary}[Maximal sets]
\label{cor:Kmn-maximal}
Let $m,n \ge 3$.
A mutual-visibility set of $K_{m,n}$ is maximal if and only
if it is of type $(m-1,n-1)$, $(m,1)$ or $(1,n)$.
\end{corollary}

\begin{proof}
Adding either of the two omitted vertices to a set of type $(m-1,n-1)$ fills one partite class while the other contains at least two tokens.
A set of type $(m,1)$ cannot receive another vertex of $B$, and the case $(1,n)$ is symmetric.
Hence all three types are maximal.

Conversely, let $S$ be maximal.
Assume first that $a<m$ and $b<n$.
If $a \le m-2$, a vertex of $A \setminus S$ may be added while both partite classes remain incomplete, contrary to maximality.
Hence $a = m-1$, and, by symmetry, $b = n-1$.
If $b = n$, Theorem~\ref{thm:Kmn-characterization} gives $a \le 1$, and maximality then gives $a=1$.
The case $a=m$ similarly gives $b=1$.
\end{proof}

Thus, maximal mutual-visibility sets need not be maximum.
For instance, a set of type $(1,5)$ in $K_{4,5}$ is maximal and has cardinality $6<\mu(K_{4,5})=7$.

For a statement $\mathcal P$, write $\mathbf{1}_{\mathcal P}=1$ if $\mathcal{P}$ is true and $0$ otherwise.

\begin{corollary}[Enumeration]
\label{cor:Kmn-count}
For $m,n \ge 2$ and $0 \le k \le m+n$, the number $r_k = \lvert \M_k(K_{m,n}) \rvert$ is
\begin{displaymath}
    \begin{split}
        r_k = \binom{m+n}{k} - \mathbf{1}_{2 \le k-n \le m}\binom{m}{k-n} - \mathbf{1}_{2 \le k-m \le n}\binom{n}{k-m} + \mathbf{1}_{k = m+n}.
    \end{split}
\end{displaymath}
\end{corollary}

\begin{proof}
Fix $k$. By Theorem~\ref{thm:Kmn-characterization}, the nonfeasible $k$-subsets are the union of two families: $\mathcal F_B$, consisting of those containing $B$ and at least two vertices of $A$, and $\mathcal F_A$, defined symmetrically.

A set in $\mathcal F_B$ is obtained by choosing $k-n$ vertices from $A$, provided $2 \le k-n \le m$.
Thus, $\lvert \mathcal F_B \rvert = \mathbf{1}_{2\le k-n\le m} \binom{m}{k-n}$.
Similarly, $\lvert \mathcal F_A \rvert = \mathbf{1}_{2\le k-m\le n} \binom{n}{k-m}$.

The intersection contains only $A\cup B$ when $k=m+n$ and is empty otherwise.
Hence $\lvert \mathcal F_A \cap \mathcal F_B \rvert = \mathbf{1}_{k=m+n}$.
Applying inclusion-exclusion to the $\binom{m+n}{k}$ subsets gives the formula.
\end{proof}

For the remainder of this section, assume that $m,n \ge 3$.  
Then Theorem~\ref{thm:cited-bipartite-mu} gives $\mu(K_{m,n})=m+n-2$.
In particular, the number of maximum mutual-visibility sets is
\begin{displaymath}
    \left\lvert \M_{m+n-2}(K_{m,n}) \right\rvert =
    \begin{cases}
        mn,   & \text{if $m,n \ge 4$,}\\
        3n+3, & \text{if $m = 3$ and \(m < n\),}\\
        3m+3, & \text{if $n = 3$ and \(n < m\),}\\
        15,   & \text{if $m = n = 3$.}
    \end{cases}
\end{displaymath}

For a fixed integer $0 \le k \le \mu(K_{m,n})$, consider the integer interval $J_k = \left[ \max \{0,k-n\}, \min \{m,k\} \right]$.
A $k$-element subset $S \subseteq V(K_{m,n})$ of type $(a,b)$ satisfies $b=k-a$.
The bounds $0\le a\le m$ and $0\le k-a\le n$ show that $J_k$ is precisely the interval of values of a allowed by the cardinality constraints, before imposing mutual visibility.

By Theorem~\ref{thm:Kmn-characterization}, such a set fails to be a mutual-visibility set precisely when $b=n$ and $a \ge 2$, or when $a=m$ and $b \ge 2$.
The first case excludes the lower endpoint $a=k-n$ when $k-n\ge2$.
The second excludes the upper endpoint $a=m$ when $k-m\ge2$.
No other values are excluded.
Therefore, the feasible values form an integer interval, which we denote by $I_k$.
Thus, $I_k = \left\{ \lvert S \cap A \rvert \colon S \in \M_k(K_{m,n}) \right\}$. Every $i\in I_k$ is realized by a mutual-visibility set: choose $X\subseteq A$ and $Y\subseteq B$ with $\lvert X\rvert=i$ and $\lvert Y\rvert=k-i$. The bounds of $J_k$ guarantee these choices. The type $(i,k-i)$ avoids the excluded endpoints, so Theorem~\ref{thm:Kmn-characterization} gives $X\cup Y\in\M_k(K_{m,n})$. 

For $i\in I_k$, the configurations with $\lvert S\cap A\rvert=i$ form level $i$ of $\R^k(K_{m,n})$.
Every configuration at level $i$ is of type $(i,k-i)$, and the construction above shows that this level is nonempty.
For two consecutive levels $q$ and $q+1$, we call $q$ the lower level and $q+1$ the upper level.

\begin{lemma}[Two consecutive levels]
\label{lem:two-levels}
Let $m,n\ge3$ and $0\le k\le m+n-2$.
Suppose that $q,q+1\in I_k$.
The subgraph of $\R^k(K_{m,n})$ induced by the configurations at levels $q$ and $q+1$ is connected.
\end{lemma}

\begin{proof}
For a configuration $S$ at level $q$, write $S = X \cup Y$, where $X = S \cap A$ and $Y = S \cap B$.
Since $q,q+1\in I_k$, we have $q<m$ and $1 \le k-q \le n$.
Thus, both $A \setminus X$ and $Y$ are nonempty.

To replace $x \in X$ by $x' \in A \setminus X$, choose $y \in Y$ and slide first from $y$ to $x'$ and then from $x$ to $y$.
Likewise, to replace $y\in Y$ by $y'\in B\setminus Y$, choose $x' \in A \setminus X$ and first slide the token at $y$ to $x'$, and then slide the token at $x'$ to $y'$.
Each intermediate configuration is at level $q+1$, and each final configuration is at level $q$.
Both levels are feasible, and all destinations are unoccupied, so the slides are admissible.

Configurations at level $q$ have equal cardinalities in each part, so successive within-part exchanges transform any one into any other.
Hence any two configurations at level $q$ can be connected by a walk using only levels $q$ and $q+1$.

Finally, every configuration at level $q+1$ contains a token in $A$ and has $k-q-1<n$ tokens in $B$.
Sliding one of these tokens to an unoccupied vertex of $B$ produces a configuration at level $q$.
Hence, the subgraph induced by the two levels is connected.
\end{proof}

\begin{theorem}[Connectivity of the layers]
\label{thm:Kmn-connectivity}
Let $m,n \ge 3$ and $0 \le k \le m+n-2$.
If $k < m+n-2$, then $\R^k(K_{m,n})$ is connected.
The maximum layer $\R^{m+n-2}(K_{m,n})$ is connected if and only if $\min\{m,n\}=3$.
If $m,n \ge 4$, this layer consists of $mn$ isolated vertices.
\end{theorem}

\begin{proof}
The layer $\R^0(K_{m,n})$ has one vertex.
Let $1\le k\le m+n-3$.
Recall that $I_k$ is obtained from $J_k$ by imposing mutual visibility.

We first show that $I_k$ has at least two elements.
If $k \le 2$, neither exclusion condition $k-n \ge 2$ nor $k-m \ge 2$ holds, so $I_k = J_k = [0,k]$.
If $k \ge 3$, then $\lvert J_k \rvert = \min \{m,k\}-\max \{0,k-n\}+1 = \min \{k,m,n,m+n-k\}+1 \ge 4$.
Only the lower endpoint with $k-n\ge2$ and the upper with $k-m\ge2$ are excluded.
Thus, at most two elements are removed, so at least two consecutive feasible values remain.

By Lemma~\ref{lem:two-levels}, every pair of consecutive levels in $I_k$ induces a connected subgraph.
Successive such subgraphs share a nonempty level.
Since $I_k$ is an integer interval, their union is the whole layer $\R^k(K_{m,n})$, which is therefore connected.

Now let $k=m+n-2$.
If $m,n\ge 4$, Theorem~\ref{thm:Kmn-characterization} implies that the only feasible type is $(m-1,n-1)$.
A token slide would produce type $(m-2,n)$ or $(m,n-2)$, neither of which is feasible.
Hence, no admissible token slide is possible.
There are $mn$ such configurations, one for every choice of an omitted vertex from each part.
Thus, the maximum layer consists of $mn$ isolated vertices.

Finally, suppose that $\min \{m,n\} = 3$.
By symmetry, assume $m=3$.
If $n>3$, the feasible types are $(1,n)$ and $(2,n-1)$, so $I_k = \{1,2\}$.
If $n=3$, the feasible types are $(1,3)$, $(2,2)$ and $(3,1)$, so $I_k=\{1,2,3\}$.
In both cases, Lemma~\ref{lem:two-levels} connects the consecutive levels and hence the entire maximum layer.
\end{proof}

Lemma~\ref{lem:two-levels} also bounds distances in connected layers using TS moves.

\begin{proposition}[Distance bound within a connected layer]
\label{prop:bipartite-distance}
Let $m,n \ge 3$ and $0 \le k \le m+n-2$.
Assume that either $k < m+n-2$, or $k = m+n-2$ and $\min\{m,n\} = 3$.
Then, for all $S,T \in \M_k(K_{m,n})$, we have $d_{\R^k(K_{m,n})}(S,T) \le \lvert S \triangle T \rvert - \lvert a_S - a_T \rvert$, where $a_S = \lvert S \cap A \rvert$ and $a_T = \lvert T \cap A \rvert$.
In particular, $d_{\R^k(K_{m,n})} (S,T) \le \lvert S \triangle T \rvert$.
Consequently, $\diam \left(\R^k(K_{m,n})\right) \le 2 \min \{k,m+n-k\}$.
\end{proposition}

\begin{proof}
The assertion is immediate when $S=T$
Set $h=\lvert S\triangle T\rvert/2$ and $d=\lvert a_S-a_T\rvert$.
Perform $d$ TS moves.
At each step, if the current configuration has type $(a,k-a)$ with $a>a_T$, move a token from a selected vertex in $A\setminus T$ to an unoccupied vertex of $T\cap B$; when $a<a_T$, interchange $A$ and $B$.
The required vertices exist because the current configuration has more selected vertices than $T$ in the source part and fewer in the destination part.
Each move changes the first coordinate by one toward $a_T$, keeping it between $a_S$ and $a_T$.
Since $I_k$ is an interval, every intermediate configuration is feasible.

After $d$ slides, the type is $(a_T,k-a_T)$, leaving only within-part exchanges.
If discrepancies remain, Theorem~\ref{thm:Kmn-connectivity} and its proof give $\lvert I_k\rvert\ge2$.
Since this integer interval has at least two elements, each feasible value has an adjacent feasible value, so the current level has an adjacent feasible level.
Lemma~\ref{lem:two-levels} realizes each exchange in two slides through that level.
The remaining $h-d$ exchanges give total length at most $d+2(h-d) = 2h-d = \lvert S \triangle T \rvert - d$.
Finally, two $k$-subsets of an $(m+n)$-set differ in at most $2 \min\{k,m+n-k\}$.
\end{proof}

The following corollary recovers the case $m,n\ge3$ of the formula for
complete bipartite graphs proved by \citet[Proposition~3.6~(i)]{DettlaffEtAl2025}.
We give a proof from layer connectivity and the component criterion.

\begin{corollary}[\citealp{DettlaffEtAl2025}]
\label{cor:Kmn-mobile}
For $m,n \ge 3$, we have
\begin{displaymath}
\Mob(K_{m,n}) =
\begin{cases}
    m+n-2, & \text{if $\min \{m,n\} = 3$},\\
    m+n-3, & \text{if $m,n \ge 4$}.
\end{cases}
\end{displaymath}
\end{corollary}

\begin{proof}
Assume first that $\min\{m,n\} = 3$.
The maximum layer is connected, and its configurations together cover both partite classes.
It follows from Corollary~\ref{cor:connected-layer-mobile} that $\Mob(K_{m,n})=\mu(K_{m,n})=m+n-2$.

Now assume that $m,n\ge  4$.
Every maximum configuration is isolated and omits one vertex from each partite class; hence, no maximum set is mobile.
On the other hand, the layer of size $m+n-3$ is connected.  
Its configurations cover $V(K_{m,n})$: every vertex belongs, for example, to a configuration of type $(m-1,n-2)$.
By Corollary~\ref{cor:connected-layer-mobile}, these configurations are mobile.
Consequently, $\Mob(K_{m,n}) = m+n-3$.
\end{proof}

The preceding characterization also yields formulas for two other mutual-visibility parameters.
For comparison, we first recall the mutual-visibility polynomial in the balanced case.

\begin{proposition}[\citealp{BujtasKlavzarTian2025}]
\label{prop:cited-balanced-polynomial}
For $n\ge3$, the complete bipartite graph $K_{n,n}$ has visibility polynomial $P_{K_{n,n}}(z)=\big((1+z)^n-z^n\big)^2+2nz^{n+1}+2z^n$.
\end{proposition}

We now extend this polynomial formula to arbitrary $m,n\ge2$.

\begin{corollary}[Visibility polynomial of $K_{m,n}$]
\label{cor:visibility-polynomial-complete-bipartite}
Let $m,n\ge2$. The visibility polynomial of $K_{m,n}$ is
$\displaystyle P_{K_{m,n}}(z)=(1+z)^{m+n}
-z^n\big((1+z)^m-1-mz\big)
-z^m\big((1+z)^n-1-nz\big)
+z^{m+n}$.
\end{corollary}

\begin{proof}
Multiply the formula in Corollary~\ref{cor:Kmn-count} by $z^k$ and sum over $0 \le k \le m+n$.
The first term sums to $(1+z)^{m+n}$.
For the first correction term, the binomial theorem gives $\sum_{k=n+2}^{m+n} \binom{m}{k-n} z^k = z^n \sum_{j=2}^{m} \binom{m}{j} z^j = z^n ((1+z)^m-1-mz)$.
The second correction term follows by interchanging $m$ and $n$, and the intersection term contributes $z^{m+n}$.
This proves the formula.
\end{proof}

For $m=n\ge3$, Corollary~\ref{cor:visibility-polynomial-complete-bipartite} simplifies to Proposition~\ref{prop:cited-balanced-polynomial}.
\citet{BujtasKlavzarTian2025} observe that the general complete bipartite case can be handled similarly, but do not state an explicit formula.
Corollary~\ref{cor:visibility-polynomial-complete-bipartite} makes the general formula explicit.

Recall that the lower mutual-visibility number $\mu^{-}(G)$ is the minimum cardinality of a maximal mutual-visibility set of $G$ \citep{BresarYero2024}.

\begin{corollary}[\citealp{BresarYero2024}, Proposition~6]
\label{cor:lower-mutual-visibility-complete-bipartite}
Let $m,n \ge 3$. Then, we have $\mu^{-}(K_{m,n}) = \min \{m,n\}+1$.
\end{corollary}

\begin{proof}
Corollary~\ref{cor:Kmn-maximal} gives the maximal types $(m-1,n-1)$, $(m,1)$ and $(1,n)$.
Their cardinalities are $m+n-2$, $m+1$ and $n+1$, respectively.
Since $m,n \ge 3$, the minimum of these three numbers is $\min \{m,n\}+1$.
\end{proof}

Theorem~\ref{thm:Kmn-connectivity} and Proposition~\ref{prop:bipartite-distance} give the main reconfiguration results for $K_{m,n}$.
The type characterization supplies their feasible levels, while the counting formulas complement earlier visibility results.
For trees, branching vertices can also disconnect intermediate layers.

\section{Trees}

Let $T$ be a tree with leaf set $L(T)$.
A \emph{branching vertex} has degree at least three.
For every nonempty set $S \subseteq V(T)$, let $H_T(S)$ be the unique minimal subtree of $T$ containing $S$.
We recall the following characterization, writing $H_T(S)$ for the Steiner subtree used in the source.

\begin{lemma}[\citealp{TonnyShikhi2026}]
\label{prop:tree-characterization}
Let $T$ be a tree and let $S \subseteq V(T)$ with $\lvert S \rvert \ge 2$.
Then $S$ is a mutual-visibility set of $T$ if and only if $S$ is precisely the leaf set $L(H_T(S))$.
\end{lemma}

We also recall the extremal value for trees.

\begin{corollary}[\citealp{DiStefano2022}]
\label{cor:cited-tree-mu}
Let $T$ be a tree with at least two vertices.
Then $L(T)$ is a mutual-visibility set and $\mu(T)=\lvert L(T)\rvert$.
\end{corollary}

Branching vertices can obstruct connectivity in intermediate layers.
Let $p$ be a vertex of degree $r \ge 3$, and fix an ordering $C_1,\ldots,C_r$ of the components of $T-p$.
We call these components the \emph{branches} of $T$ at $p$.
For $S\subseteq V(T)\setminus\{p\}$, define the \emph{occupancy vector} of $S$ at $p$ with respect to this ordering, as $\alpha_p(S) = \left( \lvert S \cap V(C_1) \rvert,\ldots,\lvert S \cap V(C_r) \rvert \right)$.
We say that the branch $C_i$ is \emph{occupied} by $S$ if $S \cap V(C_i) \neq \emptyset$.

The next lemma shows that $p$ acts as a barrier whenever at least three branches are occupied: $p$ stays unoccupied and the occupancy vector is invariant.

\begin{lemma}[Branch barrier]
\label{lem:branch-barrier}
Let $T$ be a tree, let $p$ be a branching vertex, and let $S\in\M_k(T)$ with $p\notin S$.
Suppose that at least three coordinates of $\alpha_p(S)$ are positive.
For $S' \in \M_k(T)$, if $S'$ belongs to the same connected component of $\R^k(T)$ as $S$, then $p \notin S'$ and $\alpha_p(S') = \alpha_p(S)$.
\end{lemma}

\begin{proof}
No admissible TS move starting at $S$ can place a token on $p$.
Such a slide from $x\in S$ would give $U=(S\setminus\{x\})\cup\{p\}$.
Since at least three branches at $p$ are occupied by $S$, there exist vertices $u,v \in S \setminus \{x\}$ in two distinct components of $T-p$ other than the component containing $x$.
The unique $uv$-path in $T$ contains $p$ as an internal vertex.
Since $u,v,p\in U$, $u$ and $v$ are not mutually visible in $U$.
Thus the slide is not admissible.

Every admissible slide therefore avoids $p$.
Its edge lies in one component $C_i$ of $T-p$.
It moves a token within $C_i$, preserving the number of tokens in every component of $T-p$.
Hence it leaves $p$ unoccupied and preserves $\alpha_p(S)$.

In particular, at least three branches at $p$ remain occupied, so the same argument applies after each admissible slide.
By induction on the slide sequence, every reachable configuration $S'$ satisfies $p\notin S'$ and $\alpha_p(S')=\alpha_p(S)$.
\end{proof}

We next identify the branching vertices that satisfy the hypotheses of Lemma~\ref{lem:branch-barrier}.
Let $S \in \M_k(T)$ with $k \ge 3$, and define its \emph{set of active branching vertices} by
\[
B_T(S) = \{v \in V(H_T(S)) \colon \deg_{H_T(S)}(v) \ge 3\}.
\]
By Lemma~\ref{prop:tree-characterization}, $S = L(H_T(S))$, so $S \cap B_T(S) = \emptyset$.
It is nonempty because a tree with at least three leaves has a branching vertex.
By minimality of $H_T(S)$, at least three branches at each $b\in B_T(S)$ contain a vertex of $S$.
Thus Lemma~\ref{lem:branch-barrier} applies at every such vertex $b$.
For each $s\in S$, let $T_s$ be the component of
$T-B_T(S)$ containing $s$.

\begin{lemma}[Pendant-branch structure]
\label{lem:tree-pendant-branches}
Let $T$ be a tree and let $S\in\M_k(T)$, where $k\ge3$.
Let $K$ be the minimal subtree of $T$ containing \(B_T(S)\).
If $B_T(S)=\{b\}$, then $K$ consists of the single vertex $b$. Every component of $T-B_T(S)$ contains at most one vertex of $S$.
Moreover, for each $s \in S$, the component $T_s$ has a unique neighbor $b_s \in B_T(S)$, and $H_T(S) = K \cup \bigcup_{s \in S} P_{b_s s}$, where $P_{xy}$ denotes the unique $xy$-path in $T$.
If we choose $x_s \in V(T_s)$ for each $s \in S$, and put $X = \{x_s \colon s \in S\}$, then
\begin{displaymath}
    H_T(X) = K \cup \bigcup_{s \in S} P_{b_s x_s} \quad \text{ and } \quad  L(H_T(X)) = X.
\end{displaymath}
\end{lemma}

\begin{proof}
For three vertices of a tree, their median is the unique vertex that lies on each of the three pairwise paths.
Here $S=L(H_T(S))$ and $B_T(S)\ne\emptyset$.
Suppose that a component $D$ of $T-B_T(S)$ contains distinct vertices $s,t \in S$, and choose $u \in S \setminus \{s,t\}$.
Let $m$ be the median of $s,t,u$.
As distinct leaves of $H_T(S)$, none of $s,t,u$ lies between the others.
The three paths from $m$ therefore start with distinct edges, so $\deg_{H_T(S)}(m)\ge3$.
On the other hand, $m$ lies on the $st$-path, which is contained in $D$.
Thus $m\in B_T(S)\cap V(D)$, a contradiction.
Hence every component of $T-B_T(S)$ contains at most one vertex of $S$.

Since $T$ is connected and $B_T(S)$ is nonempty, each $T_s$ has at least one neighbor in $B_T(S)$.
Suppose that it has distinct neighbors $b,b' \in B_T(S)$.
The internal vertices of the $bb'$-path lie in $T_s$, and this path is contained in $H_T(S)$.
If $s$ lies on this path, then it is an internal vertex and cannot be a leaf of $H_T(S)$, a contradiction.
Otherwise, let $m$ be the median of $s,b,b'$.
Since $T_s$ is connected and contains $s$ and the internal vertices of the $bb'$-path, the path from $s$ to the $bb'$-path meets it at a vertex $m \in V(T_s)$.
The paths from $m$ to $s$, $b$, and $b'$ start with three distinct edges of $H_T(S)$, so $\deg_{H_T(S)}(m) \ge 3$.
Thus $m \in B_T(S) \cap V(T_s)$, again a contradiction.
Hence $T_s$ has a unique neighbor $b_s \in B_T(S)$.

Both $K$ and the paths $P_{b_s s}$ lie in $H_T(S)$.
Their union is connected and contains $S$.
By the minimality of $H_T(S)$, we have $H_T(S) = K \cup \bigcup_{s \in S} P_{b_s s}$.

Now choose $x_s \in V(T_s)$ for every $s \in S$ and put $X = \{x_s \colon s \in S\}$.
For each $b\in B_T(S)$, the vertices $s$ and $x_s$ lie in the same component of $T-b$.
Thus $X$ occupies exactly the same branches at $b$ as $S$.
In particular, every $b \in B_T(S)$ still has degree at least three in $H_T(X)$, and hence $K \subseteq H_T(X)$.
Since $H_T(X)$ contains both $b_s$ and $x_s$, it also contains every path $P_{b_s x_s}$.
Conversely, this union is connected and contains $X$.
Therefore, $H_T(X) = K \cup \bigcup_{s \in S} P_{b_s x_s}$.
Each $x_s$ is the endpoint of its corresponding pendant path, so $X \subseteq L(H_T(X))$.
By minimality, $L(H_T(X))\subseteq X$, so $L(H_T(X))=X$.
\end{proof}

Recall that the Cartesian product $\mathop{\square}_{i \in I} G_i$ has vertex set $\prod_{i \in I} V(G_i)$.
Two tuples are adjacent if they differ in exactly one coordinate, along an edge of its factor $G_i$.

\begin{theorem}[Active-branch decomposition]
\label{thm:tree-component-decomposition}
Let $T$ be a tree and let $S\in\M_k(T)$, where $k\ge3$.
Then the subtrees $T_s$, $s\in S$, are pairwise vertex-disjoint.
For each tuple $(x_s)_{s \in S}$, with $x_s \in V(T_s)$ for every $s \in S$, define $\varphi\left((x_s)_{s \in S}\right) = \{x_s \colon s \in S\}$.
The map $\varphi$ is a graph isomorphism
\begin{displaymath}
    \mathop{\square}_{s\in S}T_s \cong \mathcal C_k(S),
\end{displaymath}
where $\mathcal C_k(S)$ is the component of $\R^k(T)$ containing $S$.
Equivalently, a configuration belongs to $\C_k(S)$ if and only if it has exactly
one token in each $T_s$ and none outside these subtrees.
\end{theorem}

\begin{proof}
By Lemma~\ref{prop:tree-characterization}, $S=L(H_T(S))$.
For $b\in B_T(S)$, every component of $H_T(S)-b$ contains a leaf, so at least $\deg_{H_T(S)}(b)\ge3$ branches are occupied.
By Lemma~\ref{lem:branch-barrier}, $B_T(S)$ stays unoccupied along every walk from $S$.
Hence tokens cannot move between components of $T-B_T(S)$.
By Lemma~\ref{lem:tree-pendant-branches}, the subtrees $T_s$ are pairwise vertex-disjoint, each initially contains one token, and all other components are unoccupied.
These properties are therefore invariant along every path in $\mathcal C_k(S)$.

Conversely, choose $x_s \in V(T_s)$ for each $s\in S$, and put $X = \varphi \left((x_s)_{s \in S}\right) = \{x_s \colon s \in S\}$.
Lemma~\ref{lem:tree-pendant-branches} and Lemma~\ref{prop:tree-characterization} imply that $X$ is a mutual-visibility set.
Move the token initially at $s$ to $x_s$ along the unique path in $T_s$, processing the tokens one at a time..
The intermediate configurations have this form and are feasible, so $X$ is reachable.
Thus $\varphi$ is well defined and surjective onto $V(\mathcal C_k(S))$.
It is injective since $x_s$ is the unique vertex in $X\cap V(T_s)$.
Adjacency means changing one coordinate along an edge of its subtree.
Hence $\varphi$ is a graph isomorphism.
\end{proof}

Figure~\ref{fig:branch-barrier-decomposition} illustrates the Cartesian-product decomposition.
In panel~(a), the branches $C_1,C_2,C_3$ contain $3$, $1$, and $2$ vertices of $S$, respectively, so $\alpha_p(S) = (3,1,2)$ in this ordering.
Lemma~\ref{lem:branch-barrier} keeps $p$ unoccupied and $\alpha_p(S)$ invariant, preventing tokens from changing branches.

In panel~(b), $B_T(S) = \{b_1,b_2\}$.
Deleting them leaves four occupied components $T_{s_i}$ with $S\cap V(T_{s_i})=\{s_i\}$, $1\le i\le4$, and the unoccupied central vertex.
The core $K$ is the path joining $b_1b_2$-path, and $H_T(S)$ consists of $K$ together with the pendant paths $P_{b_{s_i}s_i}$.
The component consists of the configurations $\{x_{s_1},x_{s_2},x_{s_3},x_{s_4}\}$ with $x_{s_i}\in V(T_{s_i})$.
An admissible token slide changes one coordinate along an edge of its subtree, giving $\mathcal C_k(S) \cong T_{s_1} \mathbin{\square} T_{s_2} \mathbin{\square} T_{s_3} \mathbin{\square} T_{s_4}$.

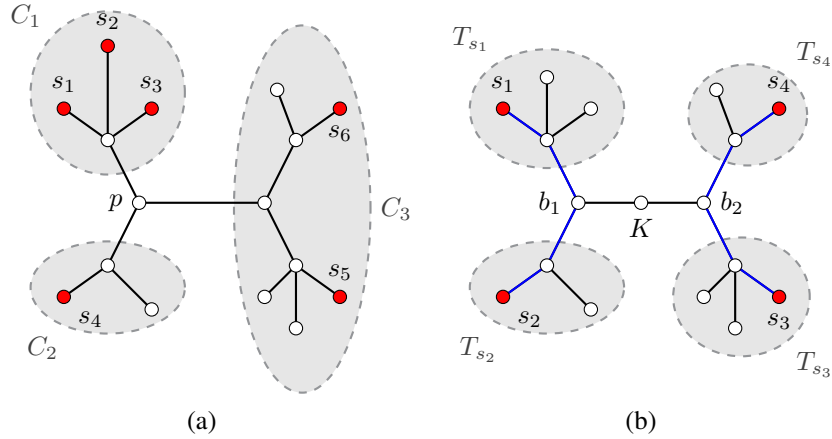
\begin{figure}[h]
\centering
\begin{tikzpicture}[scale = 0.83,
    vertex/.style = {
        circle,
        draw = black,
        fill = white,
        inner sep = 0pt,
        minimum size = 5pt
    },
    redvertex/.style = {
        circle,
        draw = black,
        fill = red,
        inner sep = 0pt,
        minimum size = 5pt
    },
    edge/.style = {
        thick,
        color = black
    },
    component/.style = {
        ellipse,
        thick,
        dashed,
        draw = Gray,
        fill = lightgray!40,
        inner sep = 4pt
    },
    arrow/.style = {
        -{Latex[length = 2.5mm]},
        draw = Black,
        line width = 0.9pt
    }
]
    \coordinate (K) at (-4,0);
    \coordinate (A) at (-4,-3.5);
    \coordinate (b1) at (-5,0);
    \coordinate (b2) at (-3,0);
    \coordinate (s1) at (-6.2,1.5);
    \coordinate (aux1) at (-5.5,1);
    \coordinate (leaf11) at (-5.5,2.5);
    \coordinate (leaf1) at (-4.8,1.5);
    \coordinate (s2) at (-6.2,-1.5);
    \coordinate (aux2) at (-5.5,-1);
    \coordinate (leaf2) at (-4.8,-1.7);
    \coordinate (s3) at (-1.8,-1.5);
    \coordinate (aux3) at (-2.5,-1);
    \coordinate (leaf33) at (-2.5,-2);
    \coordinate (leaf3) at (-3,-1.5);
    \coordinate (s4) at (-1.8,1.5);
    \coordinate (aux4) at (-2.5,1);
    \coordinate (leaf4) at (-2.8,1.8);

    \node[component, fit = (s1)(aux1)(leaf1)(leaf11), label = above left:{\textcolor{darkgray}{$C_1$}}] {};
    \node[component, fit = (s2)(aux2)(leaf2), label = below left:{\textcolor{darkgray}{$C_2$}}] {};
    \node[component, fit = (s3)(aux3)(leaf33)(leaf3)(b2)(s4)(aux4)(leaf4), label = right:{\textcolor{darkgray}{$C_3$}}] {};

    \draw[edge]
        (b1)--(b2)
        (s1)--(aux1)--(b1) (leaf11)--(aux1)--(leaf1)
        (s2)--(aux2)--(b1) (aux2)--(leaf2)
        (s3)--(aux3)--(b2) (leaf33)--(aux3)--(leaf3)
        (s4)--(aux4)--(b2) (aux4)--(leaf4)
    ;

    \node[vertex, label = left:{$p$}] at (b1) {};
    \node[vertex] at (b2) {};

    \node[redvertex, label = above:{$s_1$}] at (s1) {};
    \node[vertex] at (aux1) {};
    \node[redvertex, label = above:{$s_2$}] at (leaf11) {};
    \node[redvertex, label = above:{$s_3$}] at (leaf1) {};

    \node[redvertex, label = below right:{$s_4$}] at (s2) {};
    \node[vertex] at (aux2) {};
    \node[vertex] at (leaf2) {};

    \node[redvertex, label = above:{$s_5$}] at (s3) {};
    \node[vertex] at (aux3) {};
    \node[vertex] at (leaf33) {};
    \node[vertex] at (leaf3) {};
    \node[redvertex, label = below:{$s_6$}] at (s4) {};
    \node[vertex] at (aux4) {};
    \node[vertex] at (leaf4) {};

    \node at (A) {(a)};
    
    \coordinate (K) at (-4+5+2,0);
    \coordinate (B) at (-4+5+2,-3.5);
    \coordinate (b1) at (-5+5+2,0);
    \coordinate (b2) at (-3+5+2,0);
    \coordinate (s1) at (-6.2+5+2,1.5);
    \coordinate (aux1) at (-5.5+5+2,1);
    \coordinate (leaf11) at (-5.5+5+2,2);
    \coordinate (leaf1) at (-4.8+5+2,1.5);
    \coordinate (s2) at (-6.2+5+2,-1.5);
    \coordinate (aux2) at (-5.5+5+2,-1);
    \coordinate (leaf2) at (-4.8+5+2,-1.7);
    \coordinate (s3) at (-1.8+5+2,-1.5);
    \coordinate (aux3) at (-2.5+5+2,-1);
    \coordinate (leaf33) at (-2.5+5+2,-2);
    \coordinate (leaf3) at (-3+5+2,-1.5);
    \coordinate (s4) at (-1.8+5+2,1.5);
    \coordinate (aux4) at (-2.5+5+2,1);
    \coordinate (leaf4) at (-2.8+5+2,1.8);

    \node[component, fit = (s1)(aux1)(leaf1)(leaf11), label = above left:{\textcolor{darkgray}{$T_{s_1}$}}] {};
    \node[component, fit = (s2)(aux2)(leaf2), label = below left:{\textcolor{darkgray}{$T_{s_2}$}}] {};
    \node[component, fit = (s3)(aux3)(leaf33)(leaf3), label = below right:{\textcolor{darkgray}{$T_{s_3}$}}] {};
    \node[component, fit = (s4)(aux4)(leaf4), label = above right:{\textcolor{darkgray}{$T_{s_4}$}}] {};

    \draw[edge]
        (b1)--(K)--(b2)
        (s1)--(aux1)--(b1) (leaf11)--(aux1)--(leaf1)
        (s2)--(aux2)--(b1) (aux2)--(leaf2)
        (s3)--(aux3)--(b2) (leaf33)--(aux3)--(leaf3)
        (s4)--(aux4)--(b2) (aux4)--(leaf4)
    ;

    \draw[edge,blue,thick]
        (s1)--(aux1)--(b1)--(aux2)--(s2)
        (s4)--(aux4)--(b2)--(aux3)--(s3)
    ;

    \node[vertex, label = below:{$K$}] at (K) {};
    \node[vertex, label = left:{$b_1$}] at (b1) {};
    \node[vertex, label = right:{$b_2$}] at (b2) {};

    \node[redvertex, label = above:{$s_1$}] at (s1) {};
    \node[vertex] at (aux1) {};
    \node[vertex] at (leaf11) {};
    \node[vertex] at (leaf1) {};

    \node[redvertex, label = below right:{$s_2$}] at (s2) {};
    \node[vertex] at (aux2) {};
    \node[vertex] at (leaf2) {};

    \node[redvertex, label = below:{$s_3$}] at (s3) {};
    \node[vertex] at (aux3) {};
    \node[vertex] at (leaf33) {};
    \node[vertex] at (leaf3) {};

    \node[redvertex, label = above:{$s_4$}] at (s4) {};
    \node[vertex] at (aux4) {};
    \node[vertex] at (leaf4) {};

    \node at (B) {(b)};
\end{tikzpicture}
\caption{(a) Branch occupancy at $p$. (b) Pendant-branch decomposition. Red vertices form the mutual-visibility set $S$, dashed regions mark the indicated components, and blue edges in (b) mark the pendant paths.}
\label{fig:branch-barrier-decomposition}
\end{figure}

The decomposition also gives the basic parameters of each component.
Recall that a connected graph is a median graph if every three vertices have a unique vertex lying on a shortest path between each pair.

\begin{corollary}
\label{cor:tree-component-parameters}
Under the hypotheses of Theorem~\ref{thm:tree-component-decomposition}, put $n_s = \lvert V(T_s) \rvert$ and $m_s = \lvert E(T_s) \rvert$.
Then
\begin{displaymath}
    \lvert V(\mathcal C_k(S)) \rvert = \prod_{s \in S} n_s, \qquad \lvert E(\mathcal C_k(S)) \rvert = \sum_{s \in S} m_s \prod_{t \in S \setminus \{s\}} n_t,
\end{displaymath}
and $\diam(\mathcal C_k(S)) = \sum_{s \in S} \diam(T_s)$.
Moreover, $\mathcal C_k(S)$ is a median graph and hence a partial cube.
\end{corollary}

\begin{proof}
Each vertex of the product is a tuple containing one vertex from each factor $T_s$.
Each edge chooses an edge in one factor and vertices in the others, proving the counts.

The product distance between tuples $x=(x_s)$ and $y=(y_s)$ is $\sum_{s \in S} d_{T_s}(x_s,y_s)$: each move changes one coordinate, and following shortest paths in successive factors attains the bound.
Choosing $x_s,y_s$ at distance $\diam(T_s)$ in each factor gives the diameter formula.

The median of three product vertices is obtained by taking the tree median in each coordinate.
Thus $\mathcal C_k(S)$ is a median graph.
Every median graph is a partial cube; see \citet{HammackImrichKlavzar2011}.
\end{proof}

The mobile value for trees is part of the following characterization.

\begin{theorem}[\citealp{DettlaffEtAl2025}]
\label{thm:cited-mobile-trees}
Let $G$ be a connected graph with at least two vertices.
Then $\Mob(G)=2$ if and only if $G$ is a tree.
\end{theorem}

Theorem~\ref{thm:tree-component-decomposition} provides a structural proof of the tree case, also stated in \citet[Corollary~3.3]{DettlaffEtAl2025}.

\begin{corollary}[\citealp{DettlaffEtAl2025}]
\label{cor:mobile-mutual-visibility-trees}
If $T$ is a tree with at least two vertices, then $\Mob(T)=2$.
\end{corollary}

\begin{proof}
Section~\ref{sec:preliminaries} gives $\Mob(T)\ge2$.
Let $S\in\M_k(T)$ with $k\ge  3$.
By Theorem~\ref{thm:tree-component-decomposition}, all configurations in $\mathcal C_k(S)$ avoid the nonempty set $B_T(S)$.
Thus the configurations in $\mathcal C_k(S)$ do not cover $V(T)$.
By Proposition~\ref{prop:mobile-component}, $S$ is not mobile.
\end{proof}

The branch-barrier lemma also yields a lower bound on the number of connected components of a layer.
Fix a branching vertex $p$ of $T$, let $r=\deg_T(p)$, and order
the components $C_1,\ldots,C_r$ of $T-p$ as in the definition
of $\alpha_p$.
For each $i\in\{1,\ldots,r\}$, put
$\ell_i = \lvert L(T) \cap V(C_i) \rvert$.

For an integer $k\ge 0$, let $Q_p(k)$ be the set of integer
vectors $(a_1,\ldots,a_r)$ with at least three positive
coordinates satisfying $0\le  a_i\le \ell_i$ for every
$i\in\{1,\ldots,r\}$ and $\sum_{i=1}^{r}a_i=k$.

\begin{lemma}
\label{lemma:tree-components-lower}
Let $T$ be a tree and let $p$ be a branching vertex.
If $3\le k<\lvert L(T)\rvert$, then $\R^k(T)$ has at least $\max\{2,\lvert Q_p(k)\rvert\}$ connected components.
\end{lemma}

\begin{proof}
For each vector $\mathbf{a} = (a_1,\ldots,a_r)\in Q_p(k)$, choose sets $S_i \subseteq L(T) \cap V(C_i)$ with $\lvert S_i \rvert = a_i$, and put $S_{\mathbf{a}} = \bigcup_{i=1}^{r} S_i$.
Then $S_{\mathbf{a}}$ consists of $k$ leaves of $T$, so it is a mutual-visibility set with $\alpha_p(S_{\mathbf{a}}) = \mathbf{a}$.
Lemma~\ref{lem:branch-barrier} places distinct vectors in distinct components, since each has at least three positive coordinates.
This gives at least $\lvert Q_p(k)\rvert$ components.

For the bound of two, choose leaves in three distinct branches at $p$ and extend this selection to a set $S$ of $k$ leaves.
This is possible since each branch contains a leaf and $3\le k<\lvert L(T)\rvert$.
Choose $z \in L(T) \setminus S$, and choose $s \in S$ in a branch other than the branch containing $z$, and put $S' = (S \setminus \{s\}) \cup \{z\}$.
Both sets are mutual-visibility sets of cardinality $k$.
Since $s,z$ lie in distinct branches, $\alpha_p(S') \neq \alpha_p(S)$.
At least three branches are occupied by $S$, so Lemma~\ref{lem:branch-barrier} places it and $S'$ in distinct components.
Together these bounds give at least $\max\{2,\lvert Q_p(k)\rvert\}$ components.
\end{proof}

Branching thus disconnects intermediate layers.
We can now characterize the connected layers of a tree.

\begin{theorem}[Connectivity of tree layers]
\label{thm:tree-connectivity}
Let $T$ be a tree with at least two vertices, and let $0\le k\le\mu(T)$.
Then $\R^k(T)$ is connected if and only if $k \in \{0,1,2,\mu(T)\}$.
\end{theorem}

\begin{proof}
If $3 \le k < \mu(T) = \lvert L(T) \rvert$, then $T$ has at least four leaves and hence a branching vertex.
Lemma~\ref{lemma:tree-components-lower} implies that $\R^k(T)$ is disconnected.
This proves that the stated condition on $k$ is necessary.

Conversely, the layers of cardinality $0$, $1$, and $2$ are connected, as noted in Section~\ref{sec:preliminaries}.
It remains to consider $k = \mu(T)$.
Let $S \in \M_k(T)$.
By Lemma~\ref{prop:tree-characterization}, $S = L(H_T(S))$.
We transform $S$ into $L(T)$ by admissible slides.

If $S \neq L(T)$, then $\lvert S \rvert = \lvert L(T) \rvert$ implies that some $s \in S$ is not a leaf of $T$.
Since $s$ is a leaf of $H_T(S)$, it has a neighbor $x \in V(T) \setminus V(H_T(S))$.
Adjoining $x$ and the edge $sx$ to $H_T(S)$ gives a subtree whose leaf set is $S' = (S \setminus \{s\}) \cup \{x\}$.
This subtree is $H_T(S')$, since a tree is the minimal subtree containing all its leaves.
By Lemma~\ref{prop:tree-characterization}, $S'$ is a mutual-visibility set, so sliding the token from $s$ to $x$ is admissible.

Each such move increases $\lvert V(H_T(S)) \rvert$ by one.
Each move strictly increases $\lvert V(H_T (S)) \rvert$, and such a move exists whenever
$S \ne L(T)$.
Since $T$ is finite, the procedure terminates at $L(T)$.
Every maximum configuration thus reaches $L(T)$, proving connectivity.
\end{proof}

The decomposition is especially explicit for spiders.
A \emph{spider} is a tree with exactly one branching vertex, denoted here by $c$.
Let $r = \deg_T(c) \ge 3$.
The components of $T-c$ are its legs, numbered $1,\ldots,r$.
The $i$-th leg is a path $P_{\ell_i}$ on $\ell_i$ vertices, not including $c$.

\begin{corollary}[Layers of a spider]
\label{cor:spider-layers}
Let $T$ be a spider whose $r \ge 3$ legs have $\ell_1,\ldots,\ell_r$ vertices, respectively.
For every $3 \le k \le r$,
\begin{displaymath}
    \R^k(T) \cong \bigplus_{\substack{I \subseteq \{1,\ldots,r\} \\ \lvert I \rvert = k}} \left(\mathop{\square}_{i \in I} P_{\ell_i} \right).
\end{displaymath}
Consequently, $\R^k(T)$ has exactly $\binom{r}{k}$ connected components.
For each $k$-element subset $I\subseteq\{1,\ldots,r\}$, the corresponding component has order $\prod_{i\in I}\ell_i$ and diameter $\sum_{i\in I}(\ell_i-1)$.
\end{corollary}

\begin{proof}
By Lemma~\ref{prop:tree-characterization}, every mutual-visibility set of cardinality at least three avoids the center $c$ and contains vertices in distinct legs.
Conversely, choosing one vertex in each of any $k$ legs gives a mutual-visibility set.
For such a configuration, the active branching set is $\{c\}$, and the factors in Theorem~\ref{thm:tree-component-decomposition} are precisely the chosen legs.
Distinct choices of legs give distinct components, proving the formulas.
\end{proof}

For cycles, all fixed-cardinality layers are connected.
We prove this using a gap-balancing argument.

\section{Cycles}

Throughout this section, $C_n$ denotes the cycle of order $n$, with vertex set $V(C_n) = \{v_0,v_1,\ldots,v_{n-1}\}$, where the subscripts are taken modulo $n$; we write $\mathbb Z_n$ for the cyclic group of integers modulo $n$.
We fix the orientation from $v_i$ to $v_{i+1}$ to specify the order of the gap-vector entries.
We use the following known value, which is the cycle case of \citet[Lemma~2.8]{DiStefano2022}.

\begin{lemma}[\citealp{DiStefano2022}]
\label{lem:cited-cycle-mu}
For every $n\ge3$, the mutual-visibility number of the cycle $C_n$ is $\mu(C_n)=3$.
\end{lemma}

Since the layers of cardinality at most two are connected, it remains to study $\R^3(C_n)$.

Let $S = \{x_1,x_2,x_3\}$ be a set of three vertices of $C_n$, listed in cyclic order according to this orientation.
These vertices divide $C_n$ into three arcs whose interiors contain no vertex of $S$.
Let $g_i$ be the length of the arc from $x_i$ to $x_{i+1}$ with interior disjoint from $S$, taking subscripts modulo $3$.
We call $\gamma(S) = (g_1,g_2,g_3)$ the gap vector of $S$.
Then $g_1,g_2,g_3$ are positive integers and $g_1+g_2+g_3=n$.

The next lemma characterizes visibility by gap vectors.

\begin{lemma}
\label{lem:gap-vector-cycles}
Let $n\ge3$ and let $S$ be a set of three vertices of $C_n$ with gap vector $\gamma(S)=(g_1,g_2,g_3)$.
Then $S$ is a mutual-visibility set if and only if $g_i \le \left\lfloor n/2 \right\rfloor$, for every $i \in \{1,2,3\}$.
\end{lemma}

\begin{proof}
For $i \in \{1,2,3\}$, let $P_i$ be the arc of length $g_i$ joining $x_i$ and $x_{i+1}$ whose interior contains no vertex of $S$.
The complementary arc has length $n-g_i$ and contains the remaining vertex of $S$.
Thus $x_i,x_{i+1}$ are mutually visible exactly when $P_i$ is a geodesic.
This occurs if and only if $g_i \le n - g_i$ or, equivalently, $g_i \le \lfloor n/2 \rfloor$.
Applying the same argument to the three pairs proves the characterization.
\end{proof}

We describe the effect of token slides on gap vectors.
Let $S \in \M_3(C_n)$, with its gap vector ordered according to the fixed orientation. Reversing the orientation reorders entries but preserves visibility and TS adjacency.
Suppose that the arcs corresponding to $g_i$ and $g_j$ share an endpoint $s \in S$, and let $g_\ell$ be the remaining gap.
Sliding $s$ one edge into the $g_i$-arc shortens it by one and lengthens the $g_j$-arc by one.
Writing $\gamma(S')=(g'_1,g'_2,g'_3)$ for the resulting configuration $S'$, we have $g'_i=g_i-1$, $g'_j=g_j+1$, and $g'_\ell=g_\ell$.
The gaps still sum to $n$.
Figure~\ref{fig:gap-vector-dynamics} illustrates this operation for the transformation $(3,4,5) \mapsto (4,3,5)$.

\begin{figure}[H]
\centering
\begin{tikzpicture}[scale = 0.83,
    vertex/.style = {
        circle,
        draw = black,
        fill = white,
        inner sep = 0pt,
        minimum size = 5pt
    },
    redvertex/.style = {
        circle,
        draw = black,
        fill = red,
        inner sep = 0pt,
        minimum size = 5pt
    },
    edge/.style = {
        thick,
        color = black
    },
    arrow/.style = {
        -{Latex[length = 2.5mm]},
        draw = Black,
        line width = 0.9pt
    }
]
    \begin{scope}[shift={(-5,0)}]
        \foreach \k in {0,...,11}
            \coordinate (L\k) at ({30*\k}:1.75);

        \foreach \i/\j in {0/1,1/2,2/3}
            \draw[edge, blue] (L\i)--(L\j);
            
        \foreach \i/\j in {3/4,4/5,5/6,6/7}
            \draw[edge, ForestGreen] (L\i)--(L\j);
            
        \foreach \i/\j in {7/8,8/9,9/10,10/11,11/0}
            \draw[edge, orange] (L\i)--(L\j);
    
        \foreach \k in {0,...,11}
            \node[vertex] at (L\k) {};
        
        \node[redvertex, label = right:{$s_1$}] at (L0) {};
        \node[redvertex, label = above:{$s_2$}] at (L3) {};
        \node[vertex, fill = lightgray, label = above left:{$s_2'$}] at (L4) {};
        \node[redvertex, label = below left:{$s_3$}] at (L7) {};
        
        \node[blue] at ({45}:2.30) {$g_1=3$};
        \node[ForestGreen] at ({150}:2.40) {$g_2=4$};
        \node[orange] at ({-75}:2.22) {$g_3=5$};
    \end{scope}

    \draw[arrow] (-2.2,0)--(2.2,0);
    \node[below = 3pt] at (0,0) {$(3,4,5) \mapsto (4,3,5)$};

    \begin{scope}[shift={(5,0)}]
        \foreach \k in {0,...,11}
            \coordinate (R\k) at ({30*\k}:1.75);
        
        \foreach \i/\j in {0/1,1/2,2/3,3/4}
            \draw[edge, blue]   (R\i)--(R\j);
            
        \foreach \i/\j in {4/5,5/6,6/7}
            \draw[edge, ForestGreen]   (R\i)--(R\j);
            
        \foreach \i/\j in {7/8,8/9,9/10,10/11,11/0}
            \draw[edge, orange] (R\i)--(R\j);
        
        \foreach \k in {0,...,11} 
            \node[vertex] at (R\k) {};
            
        \node[redvertex, label = right:{$s_1$}] at (R0) {};
        \node[vertex, fill = lightgray, label = above:{$s_2$}] at (R3) {};
        \node[redvertex, label = above left:{$s'_2$}] at (R4) {};
        \node[redvertex, label = below left:{$s_3$}] at (R7) {};
        
        \node[blue]   at ({60}:2.30)  {$g'_1=4$};
        \node[ForestGreen] at ({165}:2.35) {$g'_2=3$};
        \node[orange] at ({-75}:2.22) {$g_3=5$};
    \end{scope}
\end{tikzpicture}
\caption{Gap-vector dynamics on $C_{12}$. Let $S = \{s_1,s_2,s_3\}$, with vertices listed in the fixed orientation and gap vector $\gamma(S)=(3,4,5)$. Sliding $s_2$ to the adjacent vertex $s'_2$ gives $S' = (S \setminus \{s_2\}) \cup \{s'_2\}$ and $\gamma(S') = (4,3,5)$, corresponding to $(g_1,g_2,g_3) \mapsto (g_1+1,g_2-1,g_3)$.}
\label{fig:gap-vector-dynamics}
\end{figure}
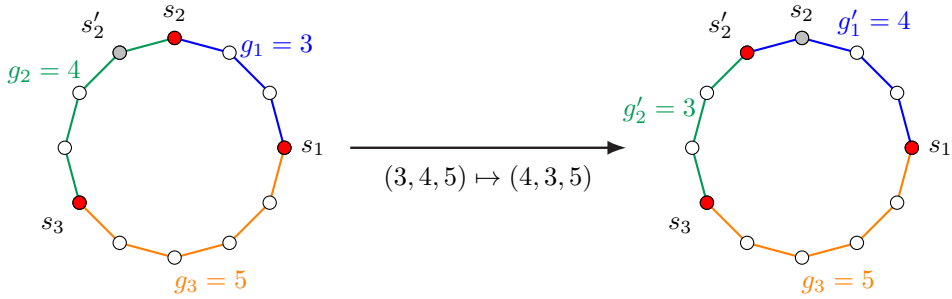

Assume now that $g_i \ge 2$ and $g_j < \lfloor n/2 \rfloor$.
The first inequality ensures an unoccupied destination and $g_i-1>0$.
The second gives $g_j+1 \le \lfloor n/2 \rfloor$.
The unchanged third gap also remains positive and at most $\lfloor n/2\rfloor$.
Lemma~\ref{lem:gap-vector-cycles} gives mutual visibility.
Therefore, the corresponding token slide is admissible.
We will use these admissible transfers between gaps in the proof below.

\begin{theorem}
For every $n \ge 3$ and every $0 \le k \le \mu(C_n)$, the layer $\R^k(C_n)$ is connected.
\end{theorem}

\begin{proof}
The layers of cardinality $0$, $1$, and $2$ are connected.
Since $\mu(C_n) = 3$, it remains to prove that $\R^3(C_n)$ is connected.

For $n=3$, the layer $\R^3(C_3)$ consists of a single vertex.
If $n=4$, then every triple in $C_4$ is mutually visible and determined by its unoccupied vertex.
A token slide moves the unoccupied vertex to a neighbor.
Consequently, $\R^3(C_4)$ is isomorphic to $C_4$.

Assume now that $n \ge 5$.
Let $S \in \M_3(C_n)$ and, after a cyclic relabeling if necessary, write $\gamma(S) = (a,b,c)$.
Each gap is at most
$\lfloor n/2 \rfloor$.

Suppose first that $n = 2t+1$.
The total capacity available in the two other gaps is $(t-b)+(t-c)=a-1$.
As long as $a>1$, at least one of $b$ and $c$ is smaller than $t$.
Transfer one unit from the $a$-gap to such a gap.
This is an admissible token slide by the preceding observation.
After $a-1$ slides, the first gap is $1$; the others are at most $t$ and sum $2t$.
Thus, the resulting gap vector is $(1,t,t)$.

Suppose now that $n=2t$.
In this case, $(t-b)+(t-c) = a$.
While $a>1$, another gap is smaller than $t$, allowing an admissible unit transfer to it.
After $a-1$ such transfers, the first gap has length $1$.
The other gaps sum to $2t-1$ and are at most $t$.
Consequently, their lengths are $t-1$ and $t$, in some order.
The gap vector is thus $(1,t-1,t)$ or $(1,t,t-1)$, up to cyclic permutation.

We now connect all configurations with these gap vectors.
First, let $n = 2t+1$ and, for $r \in \mathbb Z_n$, set $A_r = \{v_r,v_{r+1},v_{r+t+1}\}$.
These are exactly the configurations with cyclic gap pattern $(1,t,t)$, up to cyclic permutation.
In the sequences below, each arrow denotes an admissible TS move. For each $r$, the sequence
\begin{displaymath}
    A_r \rightarrow \{v_r,v_{r+2},v_{r+t+1}\} \rightarrow \{v_r,v_{r+2},v_{r+t+2}\} \rightarrow A_{r+1}
\end{displaymath}
consists of admissible token slides.
Indeed, the two intermediate gap vectors are $(2,t-1,t)$ and $(2,t,t-1)$, respectively.
Hence $A_r$ and $A_{r+1}$ lie in the same component for every $r$.
Thus all $A_r$ are mutually reachable.
Finally, let $n = 2t$, where $t \ge 3$, and define $A_r = \{v_r,v_{r+1},v_{r+t}\}$ and $B_r = \{v_r,v_{r+1},v_{r+t+1}\}$.
The gap vectors of $A_r$ and $B_r$ are $(1,t-1,t)$ and $(1,t,t-1)$, respectively.
Sliding $v_{r+t}$ to $v_{r+t+1}$ gives the edge $A_rB_r$.
Moreover, the sequence
\begin{displaymath}
    A_r \rightarrow \{v_r,v_{r+2},v_{r+t}\} \rightarrow \{v_r,v_{r+2},v_{r+t+1}\} \rightarrow A_{r+1}
\end{displaymath}
also consists of admissible token slides.
The intermediate gap vectors are $(2,t-2,t)$ and $(2,t-1,t-1)$.
Thus all configurations $A_r$ lie in one component; every $B_r$ lies in that component because it is adjacent to $A_r$.
The reduction above joins every vertex of $\R^3(C_n)$ to some $A_r$ when $n$ is odd, and to some $A_r$ or $B_r$ when $n$ is even. In each case, all configurations in these families are mutually reachable.
Thus, $\R^3(C_n)$ is connected.
\end{proof}

This completes the connectivity proof for cycles, in contrast with the branch-occupancy obstructions found for trees.

\section{Concluding remarks}

The daisy-cube TAR skeleton ensures connectivity of $\R(G)$; TS edges may shorten its distances.
We obtained metric bounds and proved that $3 \le \kappa(\R(G)) \le n$ for every connected graph $G$ of order $n \ge 3$, with both bounds attained.
Layer connectivity depends on the host.
When $m,n \ge 3$, every layer of $K_{m,n}$ is connected except the maximum layer when $m,n \ge 4$, which consists of $mn$ isolated vertices.
For trees, every component in a layer of cardinality at least three is a Cartesian product of subtrees determined by active branching vertices.
For cycles, the gap-vector argument proves that every layer is connected.

A configuration is mobile if and only if the configurations in its layer component together cover $V(G)$.
Connectivity of a layer alone does not imply this covering property.
For example, the maximum layer of $K_{1,r}$, with $r \ge 3$, consists of the single configuration formed by all leaves.
This layer is connected, but the center is never occupied along any legal move sequence, so the configuration is not mobile.

For trees, can Theorem~\ref{thm:tree-component-decomposition} yield closed formulas for layer-component counts? For general hosts, open questions concern the complexity of fixed-layer reachability, graphs with every nonempty layer connected, components covering $V(G)$, and whether $\kappa(\R(G))=\delta(\R(G))$.

\section*{AI Disclosure}

We used ChatGPT and Codex (OpenAI) to improve readability
and assist with LaTeX typesetting and figure preparation.
No AI tools contributed to the mathematical content of
the manuscript. The authors reviewed and edited all
AI-assisted output and take full responsibility for
the final manuscript.

\setlength{\bibsep}{2pt}
\bibliographystyle{abbrvnat}
\bibliography{references1}

\end{document}